\documentclass[11pt]{article}

\usepackage{amsmath}
\usepackage{amssymb}
\usepackage{amsfonts}
\usepackage[dvipdfmx]{graphicx}
\usepackage{mathtools}
\usepackage[round]{natbib}
\usepackage{fullpage}
\usepackage{float}
\usepackage{svg}

\usepackage{amsthm}
\theoremstyle{plain}

\newtheorem{thm}{Theorem}
\newtheorem{lem}{Lemma}

\newtheorem{remark}{Remark}

\theoremstyle{definition}
\newtheorem{asm}{Assumption}
\newtheorem{dfn}{Definition}
\newtheorem{exm}{Example}

\renewcommand{\d}{\,\mathrm{d}}
\let\div\relax
\DeclareMathOperator{\div}{\nabla_{\mathit{x}}\,\cdot\,}
\DeclareMathOperator{\divm}{\nabla_{\M}\,\cdot\,}
\newcommand{\A}{\mathcal{A}}
\DeclareMathOperator{\trace}{tr}
\DeclarePairedDelimiter{\ang}{\langle}{\rangle}
\DeclarePairedDelimiter{\norm}{\lVert}{\rVert}

\newcommand{\M}{\mathcal{M}}
\newcommand{\R}{\mathbb{R}}
\renewcommand{\S}{\mathbb{S}}
\newcommand{\X}{\mathfrak{X}}
\DeclareMathOperator{\E}{E}

\DeclareMathOperator{\AVar}{AVar}
\newcommand{\score}{\nabla_{x}\log q}
\newcommand{\scorem}{\nabla_{\M}\log q}
\newcommand{\smallop}[1][n^{-1/2}]{o_{\mathrm{p}}\left(#1\right)}
\newcommand{\largeop}[1][n^{-1/2}]{O_{\mathrm{p}}\left(#1\right)}

\newcommand{\warrow}{\rightsquigarrow}

\title{The geometry of Stein's method of moments:\\A canonical decomposition via score matching}
\date{}
\author{Mitsuki Nagai\thanks{The Graduate University for Advanced Studies, SOKENDAI, e-mail: \texttt{nagai.mitsuki@ism.ac.jp}} \and Keisuke Yano\thanks{The Institute of Statistical Mathematics, e-mail: \texttt{yano@ism.ac.jp}}}

\begin{document}

\maketitle

\begin{abstract}
In this paper, we elucidate the geometry of Stein's method of moments (SMoM).
SMoM is a parameter estimation method based on the Stein operator, and yields a wide class of estimators that do not depend on the normalizing constant.
We present a canonical decomposition of an SMoM estimator after centering the score matching estimator, which sheds light on the central role of the score matching within the SMoM framework.
Using this decomposition, we construct an SMoM estimator that improves upon the score matching estimator in the asymptotic variance.
We also discuss the connection between SMoM and the Wasserstein geometry.
Specifically,
using the Wasserstein score function, we provide a geometrical interpretation of the gap in the asymptotic variance between the score matching estimator and the maximum likelihood estimator.
Furthermore, it is shown that the score matching estimator is asymptotically efficient if and only if the Fisher score functions span the same space as the Wasserstein score functions.
\end{abstract}

\section{Introduction}
We consider estimation of statistical models whose probability density $q_{\theta}$ has the form
\[q_\theta(x)=\frac{1}{Z(\theta)}\tilde{q}_\theta(x),\]
where $\tilde{q}_\theta$ is a non-normalized density, and $Z(\theta)\coloneqq\int\tilde{q}_{\theta}(x)\d x$.
Such models are known as non-normalized models or energy-based models, and appear in various fields \citep[cf.][]{LeCun2006,Song2021}.
Although flexible, non-normalized models often involve a computationally intractable normalizing constant $Z(\theta)$,
and thus the maximum likelihood estimator (MLE) is computationally infeasible for non-normalized models.

\citet{Hyvarinen2005} introduced 
score matching as an estimation method for non-normalized models.
Score matching estimates the parameter by minimizing the distance between the score function $\score_{\theta}$ and the true one $\score_{\theta^\star}$.
This procedure does not depend on the normalizing constant or the true score function $\score_{\theta^\star}$ since the minimization problem can be rewritten in a form that avoids $\score_{\theta^\star}$ using integration by parts.

Recently, \citet{Ebner2025} proposed 
yet another approach for non-normalized models,
Stein's method of moments (SMoM).
This approach utilizes a Stein operator $\A_{\theta}$, an operator satisfying the identity $\E_{\theta}[\A_{\theta}f]=0$ for a class of test functions $f$ on $\R^p$.
In this paper, we allow test functions $f_{\theta}$ that may depend on $\theta$.
Since the Stein operator is not unique, we need to choose an appropriate one, and employ the divergence-based Stein operator \citep{Mijoule2023}:
\[\A_{\theta}f_\theta\coloneqq\frac{\div(q_{\theta}f_\theta)}{q_{\theta}}, \quad f_\theta\colon\R^p\to\R^p.\]
Under the boundary condition $\lim_{\norm{x}\to\infty}q_{\theta}(x)f_\theta(x)=0$, we have $\E_{\theta}[\A_{\theta}f_\theta]=0$.
This Stein operator does not depend on the normalizing constant, implying that SMoM yields a wide class of normalizing-constant-free estimators.

In this paper, we provide a geometric framework for analyzing SMoM on the basis of score matching.
We derive the asymptotic linear representation of SMoM estimators after centering the score matching estimator (Theorem~\ref{thm:smom_decomp}).
The representation reveals that, within the framework of SMoM, the relevant orthogonal structure arises from the space of test functions, 
rather than from the space of estimating functions \citep[e.g.,][]{Bickel1993}.
Using the difference of orthogonality between these spaces, we construct an SMoM estimator that improves upon the score matching estimator in the asymptotic variance (Theorem~\ref{thm:improve}),
and show the condition for the asymptotic efficiency of the score matching estimator (Theorem~\ref{thm:lowerbound}).

One intriguing aspect of our construction is its connection to the Wasserstein geometry \citep[e.g.,][]{Otto2000,Otto2001,Li2023,Amari2024,Ay2025,Nishimori2025,Trillos2025}. 
In fact, our construction of the SMoM estimator improving the score matching estimator 
admits a geometric interpretation as the approach to the space spanned by the Wasserstein score functions \citep{Li2023}.
Also, using the Wasserstein score function, we show that the gap in the asymptotic variance between the score matching estimator and the MLE also has a geometrical interpretation.
Furthermore, we show that the score matching estimator is asymptotically efficient if and only if the Fisher score functions span the same space as the Wasserstein score functions.
\subsection{Literature review and contributions}

There is an extensive literature on the estimation for non-normalized models, beginning with Hyv\"{a}rinen's pioneering work \citep{Hyvarinen2005}.
Regarding the denoising autoencoder as a non-normalized model,
\citet{Vincent2011} proposed denoising score matching.
To improve computational efficiency on score matching, \citet{Song2020} introduced sliced score matching.
The framework of scoring rules is extended to cover score matching \citep{Parry2012,Kanamori2015,Takasu2018}.
\citet{Sriperumbudur2017} proposed density estimation using infinite dimensional exponential families with score matching.
\citet{Gutmann2010} introduced yet another estimation method called noise contrastive estimation.
\citet{Gutmann2011} provided a unified framework based on the Bregman divergence that encompasses the score matching and noise contrastive estimation.
\citet{Matsuda2021} derived model selection criteria for non-normalized models.
In settings with missing data, \citet{Uehara2020} combined imputation techniques with estimators for non-normalized models, and \citet{Givens2025} utilized marginal score functions.
\citet{Uehara2020b} developed asymptotically efficient estimators that do not depend on the normalizing constant by employing density-ratio matching with a nonparametric density estimator.

Despite the broad literature, the statistical efficiency of the original score matching has not been well understood.
Using the isoperimetric constant,
\cite{Koehler2022} analyzed
the condition under which the score matching estimator is comparable to the maximum likelihood estimator.
Our work provides a further characterization within the SMoM framework, and constructs an estimator that improves upon the score matching when the latter is not asymptotically efficient.

Non-normalized models also naturally arise on more general domains, such as truncated domains in $\R^p$ (e.g., the non-negative orthant $\R_+^{p}$) or Riemannian manifolds (e.g., the unit sphere $\mathbb{S}^{p-1}$).
Score matching has been generalized to general domains by changing the metric $\ang{\,\cdot\,,\,\cdot\,}$ to a weighted metric $w\ang{\,\cdot\,,\,\cdot\,}$ \citep[e.g.,][]{Hyvarinen2007,Liu2022}, or replacing differential operators with their manifold counterparts \citep[e.g.,][]{Dawid2005,Mardia2016,Williams2022}.
SMoM also has been generalized in the same way \citep[c.f.,][]{Fischer2024,Fischer2025}.
In line with these generalizations, our results are extended to these settings;
see \ref{sec:main_result_M}.
For simplicity, we primarily focus on non-normalized models on $\R^p$ in the main text.

Our analysis is motivated by the fact that 
the score matching estimator corresponds to the SMoM estimator based on the test functions $f_{\theta,j}\coloneqq\nabla_{x}\partial_{\theta_j}\log q_{\theta}$; see Lemma~\ref{lem:smom_sm}.
This lemma has been previously mentioned by \citet{Ebner2025}, \citet{Eguchi2025}, and \citet{Kume2026}; it stands as a clue to investigate the relationship between score matching and the Stein operator, which has delivered various estimators.
\citet{Eguchi2025} extended Lemma \ref{lem:smom_sm} using the $\gamma$-divergence and developed the $\gamma$-score matching estimator, a robust estimator based on the $\gamma$-Stein operator.
\citet{Barp2019} proposed an estimation method based on Stein discrepancy and derived its relationship with score matching.
Closely related to the present paper is \citet{Kume2026}, where Lemma \ref{lem:smom_sm} together with the generalized method of moments (GMM) framework is employed to construct an estimator that improves upon score matching
in exponential families. We emphasize that our geometric construction delivers a different perspective on score matching in general statistical models, and reveals an unexpected connection between score matching and the Wasserstein geometry.

\subsection{Organization}

The rest of this paper is organized as follows.
In Section~\ref{sec:preliminaries}, we prepare notations and review the score matching and the SMoM.
In Section~\ref{sec:main_result}, we give our main results concerning the canonical decomposition of SMoM estimators and the construction of a SMoM estimator improving the asymptotic variance of the score matching estimator.
In Section~\ref{sec:Wasserstein}, we investigate the connections between the Wasserstein geometry and SMoM.
In Section~\ref{sec:simu_appl}, we present numerical experiments.
In Section~\ref{sec:concl}, we conclude the paper.
In \ref{sec:regularity}, we provide the regularity conditions for the model and test functions.
In~\ref{sec:Wasserstein_proof}, we provide the proofs for the results in Section~\ref{sec:Wasserstein}.
In~\ref{sec:main_result_M}, we give our results on more general domains.
In \ref{sec:simu_appl_M}, we present additional numerical experiments.

\section{Preliminaries}
\label{sec:preliminaries}

In this section, we prepare notations and review the score matching and the SMoM.

The standard inner product and the induced norm on $\R^p$ is denoted by $\ang{\,\cdot\,,\,\cdot\,}$ and $\norm{\,\cdot\,}$, respectively.
The gradient operator on $\R^{p}$ is denoted by $\nabla_{x}$,
the divergence of a vector field $f\colon\R^p\to\R^{p}$ is defined as $\div{f}\coloneqq\sum_{k=1}^{p}\partial f_k/\partial x_{k}$,
and the Laplacian of a function $g\colon\R^p\to\R$ is defined as $\Delta_{x}g\coloneqq \div(\nabla_{x} g)=\sum_{k=1}^{p}\partial^{2}{g}/\partial{x_{k}^{2}}$.
Let $\partial_{\theta_j}$ be the partial differential operator $\partial/\partial \theta_j$.

Let the parameter space $\Theta\subset\R^d$ be an open set, and let $q_{\theta}$ be a  probability density on $\R^d$ such that $q_\theta>0$ almost everywhere.
We assume that the model $\{q_{\theta}\mid\theta\in\Theta\}$ is identifiable, that is, $q_{\theta}=q_{\theta'}$ if and only if $\theta=\theta'$.
For an estimator $\hat\theta$ that is asymptotically normal, i.e.~$\sqrt{n}(\hat\theta-\theta^\star)\warrow N(0,V)$, we define its asymptotic variance as $\AVar[\hat\theta]\coloneqq V$.
The Fisher information matrix $I(\theta^\star)\in\R^{d\times d}$ is defined by $I(\theta^\star)_{jk}\coloneqq\E_{\theta^\star}[(\partial_{\theta_j}\log q_{\theta^\star})(\partial_{\theta_k}\log q_{\theta^\star})]$.
We assume that $q_{\theta}$ and test functions $f_{\theta,j}\colon\R^p\to\R^p$ for $j=1,\dots,d$ are smooth with respect to both $\theta$ and $x$.

Let $X_1,\dots,X_n$ be i.i.d.~random variables from $q_{\theta^\star}$.
The score matching estimator $\hat\theta_{\mathrm{SM}}$ is given as a solution to the estimating equations:
\[\frac{1}{n}\sum_{i=1}^{n}\partial_{\theta_j}H(X_i,\theta)=0,\quad j=1,\dots,d,\]
where the Hyvärinen score $H$ is defined by \[H(x,\theta)\coloneqq\norm{\score_\theta(x)}^{2}/2 + \Delta_{x}\log q_\theta(x).\]
Under the boundary condition $\lim_{\norm{x}\to\infty}q_{\theta^\star}(x)\score_{\theta}(x)=0$, using integration by parts, we have the relation \[\E_{\theta^\star}[H(X,\theta)]=\E_{\theta^\star}[\norm{\score_{\theta^\star}-\score_{\theta}}^{2}]/2+c,\]
where $X\sim q_{\theta^\star}$ and $c$ is a constant with respect to $\theta$.
Thus, minimizing $\E_{\theta^\star}[H(X,\theta)]$ over $\Theta$ is equivalent to minimizing $\E_{\theta^\star}[\norm{\score_{\theta^\star}-\score_{\theta}}^{2}]$.

The SMoM estimator $\hat\theta_\mathrm{SMoM}$ based on the test functions $f_{\theta,1},\dots,f_{\theta,d}$ is given as a solution to the estimating equations:
\[\frac{1}{n}\sum_{i=1}^{n}\A_{\theta}f_{\theta,j}(X_{i})=0,\quad j=1,\dots,d.\]

The following lemma states that the score matching estimator is an SMoM estimator.

\begin{lem}[Section 2.1 of \cite{Eguchi2025}; Lemma 3.2 of \cite{Kume2026}]
    \label{lem:smom_sm}
    The SMoM estimator based on the test functions $f_{\theta,j}\coloneqq\nabla_{x}\partial_{\theta_j}\log q_{\theta}, \enspace j=1,\dots,d$ is the score matching estimator.
\end{lem}
\begin{proof}
    Since $f_{\theta,j}=\nabla_{x}\partial_{\theta_j}\log q_\theta$ is a solution of the equation $\A_{\theta}f_{\theta,j}(x)=\partial_{\theta_j}H(x,\theta)$,
    the estimating equations of score matching coincide with that of SMoM based on the test functions.
\end{proof}

\begin{remark}[Test functions and gradient vector fields]
    By the Helmholtz decomposition \citep{Ay2025}, a test function $f_{\theta}\colon\R^p\to\R^p$ can be decomposed a gradient term and a divergence-free term; there exist $h_\theta\colon\R^p\to\R$ and $g_\theta\colon\R^p\to\R^p$ such that 
    \[
    \A_{\theta}g_\theta=0
\quad\text{and}\quad
f_{\theta}=\nabla h_{\theta}+g_{\theta}.
    \]
    However, we do not restrict test functions to gradients because modeling vector fields directly is computationally efficient. 
\end{remark}

\section{Main results}
\label{sec:main_result}

This section presents our main results.
To state the results rigorously, we impose regularity conditions for the model and test functions detailed in \ref{sec:regularity}.
Particularly, we assume the consistency and asymptotic linearity for SMoM estimators, and they are verified by the standard theory of $Z$-estimators \citep[e.g.,][]{VanderVaart1998}.
Under these conditions and Lemma~\ref{lem:smom_sm}, the score matching estimator $\hat\theta_\mathrm{SM}$ and the SMoM estimator $\hat\theta_\mathrm{SMoM}$ based on the test functions $f_{\theta,1},\dots,f_{\theta,d}$ have the following asymptotic linear representation:
\begin{align}
\hat\theta_\mathrm{SM}-\theta^\star
    &=-G^{-1}\frac{1}{n}\sum_{i=1}^{n}\begin{pmatrix}\A_{\theta^\star}(\nabla_{x}\partial_{\theta_1}\log q_{\theta^\star})(X_i)\\\vdots\\\A_{\theta^\star}(\nabla_{x}\partial_{\theta_d}\log q_{\theta^\star})(X_i)\end{pmatrix}+ \smallop,\label{eq:SM linear}\\
\hat\theta_\mathrm{SMoM}-\theta^\star
    &=-G_\mathrm{SMoM}^{-1}\frac{1}{n}\sum_{i=1}^{n}\begin{pmatrix}\A_{\theta^\star}f_{\theta^\star,1}(X_i)\\\vdots\\\A_{\theta^\star}f_{\theta^\star,d}(X_i)\end{pmatrix} + \smallop,\label{eq:SMoM linear}
\end{align}
respectively,
where $G,G_\mathrm{SMoM}\in\R^{d\times d}$ are the matrices whose $(j,k)$-th entries are defined by
\[G_{jk}\coloneqq\E_{\theta^\star}\Big[\partial_{\theta_k}\A_{\theta}(\nabla_{x}\partial_{\theta_j}\log q_\theta)\Big|_{\theta=\theta^\star}\Big],\quad (G_\mathrm{SMoM})_{jk}\coloneqq\E_{\theta^\star}\Big[\partial_{\theta_k}\A_{\theta}f_{\theta,j}\Big|_{\theta=\theta^\star}\Big],\]
respectively.

\subsection{A canonical decomposition of SMoM estimators}

We begin with presenting a canonical decompotision of SMoM estimators (Theorem \ref{thm:smom_decomp}).
Before presenting the main theorem, we introduce two  notions of orthogonality: $W$-orthogonality and $\mathcal{A}_{\theta^{\star}}$-orthogonality. The acronym $W$ refers to Wasserstein; see Section \ref{sec:Wasserstein} for details.
\begin{dfn}
    A test function $f$ is called $W$-orthogonal to a test function $g$ if we have
    $
    \E_{\theta^\star}[\ang{f,g}]=0.
    $
    A test function $f$ is called $\mathcal{A}_{\theta^{\star}}$-orthogonal to a test function $g$ if we have
    $
    \E_{\theta^\star}[(\mathcal{A}_{\theta^{\star}}f)(\mathcal{A}_{\theta^{\star}}g)]=0.
    $
\end{dfn}

The following theorem 
provides the asymptotic linear representation of the SMoM estimator after centering the score matching estimator.

\begin{thm}[Canonical decomposition of SMoM estimator]
    \label{thm:smom_decomp}
    Under the regularity conditions, there exist $u_j^{\star}\colon\R^p\to\R^p\,(j=1,\dots,d)$ such that the asymptotic linear representation of $\hat\theta_\mathrm{SMoM}$ can be decomposed as follows: 
    \[\hat\theta_\mathrm{SMoM}-\theta^\star=\Big(\hat\theta_\mathrm{SM}-\theta^\star\Big)-G^{-1}\frac{1}{n}\sum_{i=1}^{n}\begin{pmatrix}\A_{\theta^\star}u_1^\star(X_i)\\\vdots\\\A_{\theta^\star}u_d^\star(X_i)\end{pmatrix} + \smallop , \]
    where each $u_{j}^\star\colon\R^p\to\R^p$ is $W$-orthogonal to the test function corresponding to score matching:
    \[
    \E_{\theta^\star}[\ang{u_j^\star,\nabla_{x}\partial_{\theta_k}\log q_{\theta^\star}}]=0, \quad k=1,\dots,d.
    \]
\end{thm}

Through this canonical decomposition, the asymptotic linear representation of SMoM estimators is characterized only by its $W$-orthogonal term $\A_{\theta^\star}u_k^\star$.
The score matching estimator is the SMoM estimator without $W$-orthogonal terms.

\begin{remark}[$W$-orthogonality does not imply $\A_{\theta^\star}$-orthogonality]
Importantly, the Stein operator does not preserve orthogonality; that is, 
even if $u_k^\star$ is $W$-orthogonal to $\nabla_{x}\partial_{\theta_j}\log q_{\theta^\star}$,
the $\A_{\theta^{\star}}$-orthogonal condition $\E_{\theta^\star}[\A_{\theta^\star}(\nabla_{x}\partial_{\theta_j}\log q_{\theta^\star})\A_{\theta^\star}u_k^\star]=0$ does not necessarily hold.
This give a clue to constructing an SMoM estimator that improves upon the score matching estimator;
by finding appropriate elements that are $W$-orthogonal to the test functions corresponding to the score matching, 
we can reduce the asymptotic variance.
In Section~\ref{sec:improve}, we show how to construct such orthogonal elements.
\end{remark}

To prove Theorem \ref{thm:smom_decomp}, 
we need the following lemma telling us that $G_\mathrm{SMoM}$ is an inner product matrix, and in particular, $G$ is a Gram matrix.
\begin{lem}
    \label{lem:G_mat}
    Under the regularity conditions, for the test functions $f_{\theta,1},\dots,f_{\theta,d}$, we have \[\E_{\theta^\star}\Big[\partial_{\theta_k}\A_{\theta}f_{\theta,j}\Big|_{\theta=\theta^\star}\Big]=\E_{\theta^\star}[\ang{f_{\theta^\star,j},\nabla_{x}\partial_{\theta_k}\log q_{\theta^\star}}],\quad j,k=1,\dots,d.\]
    In particular, we have $G_{jk}=\E_{\theta^\star}[\ang{\nabla_{x}\partial_{\theta_j}\log q_{\theta^\star},\nabla_{x}\partial_{\theta_k}\log q_{\theta^\star}}]$.
\end{lem}
\begin{proof}
    Observe the relation \[\partial_{\theta_k}\A_\theta f_{\theta,j}=\A_{\theta}(\partial_{\theta_k}f_{\theta,j})+\ang{f_{\theta,j},\nabla_{x}\partial_{\theta_k}\log q_\theta}.\]
    Together with $\E_{\theta^\star}[\A_{\theta^\star}(\partial_{\theta_k}f_{\theta^\star,j})]=0$ under the regularity conditions,
    taking the expectation of this yields $\E_{\theta^\star}[\partial_{\theta_k}\A_{\theta}f_{\theta,j}|_{\theta=\theta^\star}]=\E_{\theta^\star}[\ang{f_{\theta^\star,j},\nabla_{x}\partial_{\theta_k}\log q_{\theta^\star}}]$, which completes the proof.
\end{proof}

Remark that this representation of $G$ gives the Riemannian metric induced by the score matching \citep{Karakida2016}.

Now, we prove Theorem \ref{thm:smom_decomp} by combining
the asymptotic linear representations \eqref{eq:SM linear}-\eqref{eq:SMoM linear} and Lemma \ref{lem:G_mat}.

\begin{proof}[Proof of Theorem \ref{thm:smom_decomp}]

Fix $f_{\theta,1},\ldots,f_{\theta,d}$ arbitrarily.
We begin with the $W$-orthogonal decomposition:
\begin{equation}
    \label{eq:orthogonal_decomp_testfunction}
    f_{\theta^\star,j}=\sum_{k=1}^{d}B_{jk}(\nabla_{x}\partial_{\theta_k}\log q_{\theta^\star} + u_{k}^\star),
\end{equation}
where $B\in\R^{d\times d}$ is a coefficient matrix and $u_{k}^\star\colon\R^p\to\R^p$ are $W$-orthogonal elements such that $\E_{\theta^\star}[\ang{u_{k}^\star,\nabla_{x}\partial_{\theta_j}\log q_{\theta^\star}}]=0$ for $j=1,\dots,d$.
This decomposition, together with Lemma \ref{lem:G_mat}, gives the relation $G_\mathrm{SMoM}=BG$.
Importantly, the coefficient matrix $B$ must be invertible for $\hat\theta_\mathrm{SMoM}$ to admit the asymptotic linear representation.

We next substitute the relation $G_\mathrm{SMoM}=BG$ to Equation \eqref{eq:SMoM linear}.
Then, 
we have
    \begin{align*}
    \hat\theta_\mathrm{SMoM}-\theta^\star
        &=-(BG)^{-1}\frac{1}{n}\sum_{i=1}^{n}B\begin{pmatrix}\A_{\theta^\star}(\nabla_{x}\partial_{\theta_1}\log q_{\theta^\star}+u_1^\star)(X_i)\\\vdots\\\A_{\theta^\star}(\nabla_{x}\partial_{\theta_d}\log q_{\theta^\star}+u_d^\star)(X_i)\end{pmatrix}+ \smallop.
    \end{align*}
From the linearity of $\A_{\theta^\star}$,
the invertibility of $B$,
and Equation \eqref{eq:SM linear}, we further obtain
    \begin{align*}
        \hat\theta_\mathrm{SMoM}-\theta^\star &=-G^{-1}\frac{1}{n}\sum_{i=1}^{n}\begin{pmatrix}\A_{\theta^\star}(\nabla_{x}\partial_{\theta_1}\log q_{\theta^\star})(X_i)\\\vdots\\\A_{\theta^\star}(\nabla_{x}\partial_{\theta_d}\log q_{\theta^\star})(X_i)\end{pmatrix}-G^{-1}\frac{1}{n}\sum_{i=1}^{n}\begin{pmatrix}\A_{\theta^\star}u_1^\star(X_i)\\\vdots\\\A_{\theta^\star}u_d^\star(X_i)\end{pmatrix}+ \smallop\\
        &=\Big(\hat\theta_\mathrm{SM}-\theta^\star\Big)-G^{-1}\frac{1}{n}\sum_{i=1}^{n}\begin{pmatrix}\A_{\theta^\star}u_1^\star(X_i)\\\vdots\\\A_{\theta^\star}u_d^\star(X_i)\end{pmatrix} + \smallop,
    \end{align*}
    which completes the proof.
\end{proof}

\begin{exm}[Exponential families]
Consider an exponential family whose density is defined by
\begin{equation}
    \label{eq:exp_fam}
    q_{\theta}(x)=\frac{1}{Z(\theta)}\exp\Big(t(x)^{\top}\theta+b(x)\Big),\enspace Z(\theta)=\int_{\R^{p}}\exp\Big(t(x)^{\top}\theta+b(x)\Big)\d{x},
\end{equation}
where $t\colon\R^p\to\R^d$ is a sufficient statistics and $b\colon\R^p\to\R$ is a base measure.

The SMoM estimator $\hat\theta_\mathrm{SMoM}$ based on (parameter-independent) test functions $f_1,\dots,f_d$ has the closed-form solution:
\[\hat\theta_\mathrm{SMoM}=-G_{\mathrm{SMoM},n}^{-1}\frac{1}{n}\sum_{i=1}^{n}\begin{pmatrix}\div f_1(X_i) + \ang{f_1(X_i),\nabla_{x}b(X_i)}\\\vdots\\\div f_d(X_i)+\ang{f_d(X_i),\nabla_{x}b(X_i)}\end{pmatrix},\]
where $G_{\mathrm{SMoM},n}\in\R^{d\times d}$ is the empirical inner product matrix whose $(j,k)$-th entry is defined by $(G_{\mathrm{SMoM},n})_{jk} \coloneqq n^{-1}\sum_{i=1}^{n}\ang{f_j(X_i),\nabla_x t_k(X_i)}$.
Observe that the relation 
\[E_{\theta^\star}[\div{f_j}+\ang{f_j,\nabla_x b}]=\E_{\theta^\star}\left[\A_{\theta^\star}f_j - \sum_{k=1}^d \ang{f_j,\nabla_x t_k}\theta^\star_k\right]=-G_\mathrm{SMoM}\theta^\star.\]
Then, if $G_\mathrm{SMoM}$ is invertible, $\hat\theta_\mathrm{SMoM}$ is consistent and has asymptotic normality under some conditions on integrability.
The linear independence of $\A_{\theta^\star}f_1,\dots,\A_{\theta^\star}f_d$ follows from Lemma~\ref{lem:connect_two_innerprod}.

The score matching estimator $\hat\theta_\mathrm{SM}$ is the SMoM estimator based on the test functions $f_{j}=\nabla_x t_j$, and also has the closed form \citep{Hyvarinen2007}:
\[\hat\theta_\mathrm{SM}=-G_{n}^{-1}\frac{1}{n}\sum_{i=1}^{n}\begin{pmatrix}\Delta_x t_1(X_i) + \ang{\nabla_x t_1(X_i),\nabla_x b(X_i)}\\\vdots\\\Delta_x t_d(X_i)+\ang{\nabla_x t_d(X_i),\nabla_x b(X_i)}\end{pmatrix},\]
where $G_{n}\in\R^{d\times d}$ is the empirical inner product matrix whose $(j,k)$-th entry is defined by $(G_{n})_{jk} \coloneqq n^{-1}\sum_{i=1}^{n}\ang{\nabla_x t_j(X_i),\nabla_x t_k(X_i)}$.
Consider the test functions $f_j=\nabla_x t_j + u_j^\star$, where $u_j^\star$ is $W$-orthogonal to $\nabla_x t_k$ for $k=1,\dots,d$.
The SMoM estimator based on this test functions can be decomposed as follows:
\begin{align*}
    \hat\theta_\mathrm{SMoM}-\theta^\star
    &=-\big(G_n + (G_{\mathrm{SMoM},n}-G_n)\big)^{-1}\frac{1}{n}\sum_{i=1}^{n}\begin{pmatrix}\A_{\theta^\star}(\nabla_x t_1)(X_i)+\A_{\theta^\star}u_1^\star(X_i)\\\vdots\\\A_{\theta^\star}(\nabla_x t_d)(X_i)+\A_{\theta^\star}u_d^\star(X_i)\end{pmatrix}\\
    &=\Big(\hat\theta_\mathrm{SM}-\theta^\star\Big)-G^{-1}\frac{1}{n}\sum_{i=1}^{n}\begin{pmatrix}\A_{\theta^\star}u_1^\star(X_i)\\\vdots\\\A_{\theta^\star}u_d^\star(X_i)\end{pmatrix}+\smallop,
\end{align*}
which agrees with Theorem \ref{thm:smom_decomp}.
\end{exm}

\subsection{Improving asymptotic variance of score matching estimator}
\label{sec:improve}
We next construct an SMoM estimator improving upon the score matching estimator in the asymptotic variance.
Fix $\theta_0\in\Theta$ and $\tilde{v}_\alpha\colon\R^p\to\R^p,\enspace\alpha=1,\dots,K$ arbitrarily. 
Using the test function $\tilde{v}_\alpha$, we define $v_{\theta_0,\alpha}\colon\R^p\to\R^p$ by the following orthogonalization procedure:
\begin{equation}
    \label{eq:orthogonalize}
    v_{\theta_0,\alpha}\coloneqq \tilde{v}_\alpha - \sum_{j=1}^{d}\left(F_{\theta_0}G_{\theta_0}^{-1}\right)_{\alpha j}\nabla_{x}\partial_{\theta_j}\log q_{\theta_0},
\end{equation}
where $F_{\theta_0}\in\R^{K\times d}$ and $G_{\theta_0}\in\R^{d\times d}$ are inner product matrices whose $(\alpha,j)$-th entry and $(j,k)$-th entry are defined by 
\begin{align*}
    (F_{\theta_0})_{\alpha j}&\coloneqq \E_{\theta_0}\left[\ang{\tilde{v}_\alpha,\nabla_{x}\partial_{\theta_j}\log q_{\theta_0}}\right],\\
    (G_{\theta_0})_{jk}&\coloneqq \E_{\theta_0}\left[\ang{\nabla_{x}\partial_{\theta_j}\log q_{\theta_0},\nabla_{x}\partial_{\theta_k}\log q_{\theta_0}}\right],
\end{align*}
respectively.
Observe the relation
\[\E_{\theta_0}[\ang{v_{\theta_0,\alpha},\nabla_{x}\partial_{\theta_j}\log q_{\theta_0}}]=0,\quad j=1,\dots,d.\]
Denote by $\hat\theta[\theta_0]$ the SMoM estimator based on the following test functions:
\begin{equation}
    \label{eq:testfunction} 
    f_{\theta,j}\coloneqq\nabla_{x}\partial_{\theta_j}\log q_{\theta} - \sum_{\alpha=1}^{K}\left(S_{\theta_0}T_{\theta_0}^{-1}\right)_{j\alpha}v_{\theta_0,\alpha},\quad j=1,\dots,d,
\end{equation}
where $S_{\theta_0}\in\R^{d\times K}$ and $T_{\theta_0}\in\R^{K\times K}$ are inner product matrices whose $(j,\alpha)$-th entry and $(\alpha,\beta)$-th entry are defined by
\begin{align*}
    (S_{\theta_0})_{j\alpha}&\coloneqq\E_{\theta_0}\left[\A_{\theta_0}(\nabla_{x}\partial_{\theta_j}\log q_{\theta_0})\A_{\theta_0}v_{\theta_0,\alpha}\right],\\
    (T_{\theta_0})_{\alpha\beta}&\coloneqq\E_{\theta_0}\left[\left(\A_{\theta_0}v_{\theta_0,\alpha}\right)\left(\A_{\theta_0}v_{\theta_0,\beta}\right)\right],
\end{align*}
respectively.
For the matrices $F_{\theta_0},G_{\theta_0},S_{\theta_0},T_{\theta_0}$, we make the following assumption.
\begin{asm}
    \label{asm:improve}
    $F_{\theta_0},G_{\theta_0},S_{\theta_0},T_{\theta_0}$ are continuous at $\theta_0=\theta^\star$.
\end{asm}

The following theorem tells us that $\hat\theta[\theta^\star]$ improves the asymptotic variance of $\hat\theta_\mathrm{SM}$.
Furthermore, such improvement remains even when  $\hat\theta_0$ is plugged in to $\theta^\star$.
The proof is given in \ref{sec:main_result_M}.
\begin{thm}
    \label{thm:improve}
    Under the regularity conditions, we have
    \[\AVar\left[\hat\theta_{\mathrm{SM}}\right]-\AVar\left[\hat\theta[\theta^\star]\right]
    =G^{-1}S_{\theta^\star}T_{\theta^\star}^{-1}S_{\theta^\star}^{\top}G^{-1}\succeq0.\]
    The equality $\AVar[\hat\theta_{\mathrm{SM}}]=\AVar[\hat\theta[\theta^\star]]$ holds if and only if $S_{\theta^\star}=O$ with a zero matrix $O$.
    Furthermore, under Assumption~\ref{asm:improve},
    for an estimator $\hat\theta_0$ satisfying $\hat\theta_0-\theta^\star=\largeop$, we have \[\hat\theta[\hat\theta_0]-\theta^\star=\hat\theta[\theta^\star]-\theta^\star+\smallop.\]
\end{thm}

Figure~\ref{fig:schematic_diagram_improvement} illustrates the improvement procedure.
Consider the SMoM estimator $\hat\theta_\mathrm{SMoM}$ based on the following test functions:
\begin{equation}
    f_{\theta,j}\coloneqq\nabla_{x}\partial_{\theta_j}\log q_{\theta} + \sum_{\alpha=1}^{K}C_{j\alpha}v_{\theta^\star,\alpha},\quad j=1,\dots,d,
\end{equation}
where $C\in\R^{d\times K}$ is a coefficient matrix.
By Theorem~\ref{thm:smom_decomp}, the asymptotic variance $\AVar[\hat\theta_\mathrm{SMoM}]$ can be decomposed as follows:
\[\AVar\left[\hat\theta_\mathrm{SMoM}\right]=\AVar\left[\hat\theta_\mathrm{SM}\right]+G^{-1}C(T_{\theta^\star}+S_{\theta^\star}+S_{\theta^\star}^\top)C^\top G^{-1},\]
and it is minimized at $C=-S_{\theta^\star}T_{\theta^\star}^{-1}$.
In Figure~\ref{fig:schematic_diagram_improvement}, the score matching estimator, which is the SMoM estimator without $W$-orthogonal terms, is located at $c_1=c_2=0$.
When $K=1$, i.e.~$c_2=0$, 
since $\AVar[\hat\theta_\mathrm{SMoM}]$ is a quadratic function of $c_1$, $\hat\theta_\mathrm{SMoM}$ can be moved in the descent direction of $\AVar[\hat\theta_\mathrm{SMoM}]$ starting from $\hat\theta_\mathrm{SM}$, resulting in $\AVar[\hat\theta_\mathrm{SMoM}]$ being minimized at $c_1=c_1^\star$.
When $K=2$, $\AVar[\hat\theta_\mathrm{SM}]$ is further improved by additional direction $v_2^\star$, and $\AVar[\hat\theta_\mathrm{SMoM}]$ is minimized at $(c_1,c_2)=c^\star$.

\begin{figure}[H]
    \centering
    \includegraphics[width=0.75\linewidth]{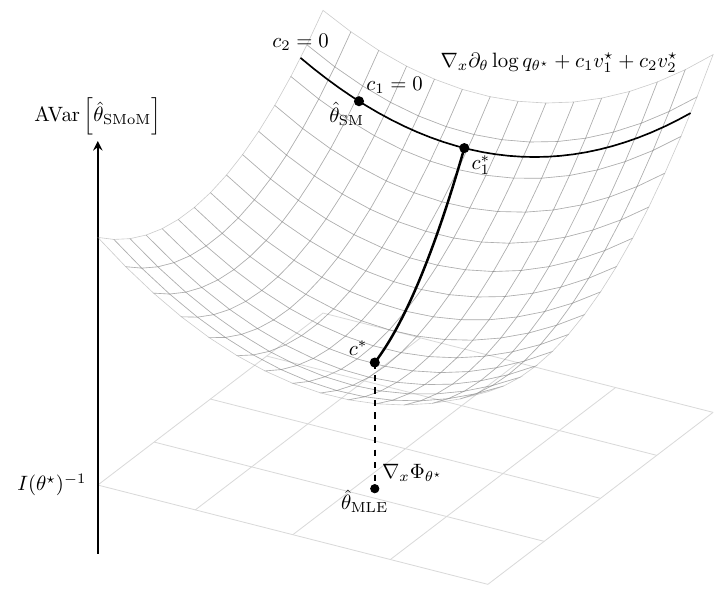}
    \caption{Schematic diagram of the improvement of the asymptotic variance $\AVar[\hat\theta_\mathrm{SM}]$ of the score matching estimator via SMoM. For $K=1$, the asymptotic variance $\AVar[\hat\theta[\theta^\star]]$ is minimized at $c_1^\star$, and it is lower than $\AVar[\hat\theta_\mathrm{SM}]$. For $K=2$, $\AVar[\hat\theta_\mathrm{SM}]$ is further improved at $c^\star$. $\AVar[\hat\theta[\theta^\star]]$ approaches the efficiency bound as increasing $K$.}
    \label{fig:schematic_diagram_improvement}
\end{figure}

Let $U_{\theta_0}\in\R^{d\times d}$ be the inner matrix whose $(j,k)$-th entry defined by
\[(U_{\theta_0})_{jk}\coloneqq\E_{\theta_0}\left[\A_{\theta_0}(\nabla_{x}\partial_{\theta_j}\log q_{\theta_0})\A_{\theta_0}(\nabla_{x}\partial_{\theta_k}\log q_{\theta_0})\right].\]
The asymptotic variance of the score matching estimator is given by $\AVar[\hat\theta_\mathrm{SM}]=G^{-1}U_{\theta^\star}G^{-1}$.
As the estimate of asymptotic relative efficiency $\AVar[\hat\theta[\hat\theta_0]]_{jj}/\AVar[\hat\theta_\mathrm{SM}]_{jj}$, we employ
\begin{equation}
    \label{eq:estimate_improvement}
    1-\frac{\left(G_{\hat\theta_0}^{-1}S_{\hat\theta_0}T_{\hat\theta_0}^{-1}S_{\hat\theta_0}^{\top}G_{\hat\theta_0}^{-1}\right)_{jj}}{\left(G_{\hat\theta_\mathrm{SM}}^{-1}U_{\hat\theta_\mathrm{SM}}G_{\hat\theta_\mathrm{SM}}^{-1}\right)_{jj}},\quad j=1,\dots,d,
\end{equation}
and this quantity can be calculated using the Monte Carlo approximation.

\begin{remark}[The choice of $\hat\theta_0,\tilde{v}_\alpha$, and $K$]
   A reasonable choice of $\hat\theta_0$ and $\tilde{v}_\alpha$ are the score matching estimator $\hat\theta_\mathrm{SM}$ and neural networks with random weight, which yields flexible and easily differentiable functions.
   Also, theoretically, increasing $K$ allows us to search for optimal $W$-orthogonal terms in a wider space. 
   In practice, we may need the Monte Carlo approximation of the expectation $\E_{\hat\theta_0}[\,\cdot\,]$, and it can lead to instability for a large $K$.
   In Section~\ref{sec:simu_appl}, we investigate this instability through numerical simulations. 
\end{remark}

\section{Connections between SMoM and the Wasserstein geometry}
\label{sec:Wasserstein}

In this section, we relate SMoM to the Wasserstein geometry.
This unexpected connection presents a further characterization of the SMoM and the score matching.

We first introduce the Wasserstein score function, 
a central concept in the Wasserstein geometry.
\begin{dfn}[The Wasserstein score function (\citealp{Li2023})]
The Wasserstein score function $\Phi_{\theta,j}\colon\R^p\to\R$ is the solution of the following partial differential equation:
\[\A_{\theta}\left(\nabla_{x}\Phi_{\theta,j}\right)=-\partial_{\theta_j}\log q_{\theta},\quad j=1,\dots,d\]
with $\E_{\theta}[\Phi_{\theta,j}]=0$.
\end{dfn}
The Wasserstein score function induces a second-order approximation of the 2-Wasserstein distance, and plays a key role in the Wasserstein geometry \citep[e.g.,][]{Li2023,Nishimori2025}.
Also, $W$-orthogonality corresponds to the orthogonality with respect to the Wasserstein covariance.

Within the SMoM framework, 
the maximum likelihood estimator is
the SMoM estimator based on the test functions $f_{\theta,j}=-\nabla_x\Phi_{\theta,j}$.
Remark that
since the Wasserstein score function often does not has a closed form and depends on the normalizing constant,
this SMoM estimator (MLE) is computationally infeasible.

Yet, regarding MLE as the SMoM from the (minus) Wasserstein score function offers the following characterization of the score matching estimator: the score matching estimator cannot be improved if and only if the Fisher score functions span the same space as the Wasserstein score functions.
In this situation, the score matching estimator coincides with the MLE.
The proof is given in \ref{sec:Wasserstein_proof}.

\begin{thm}
    \label{thm:lowerbound}
    Assume the regularity conditions.
    Then, the following are equivalent:
    \begin{enumerate}
        \item the score matching estimator is asymptotically efficient;
        \item if $u^\star$ is $W$-orthogonal to $\nabla_x\partial_{\theta_j}\log q_{\theta^\star}$ for $j=1,\dots,d$, then it is also $\A_{\theta^\star}$-orthogonal to $\nabla_x\partial_{\theta_j}\log q_{\theta^\star}$ for $j=1,\dots,d$; 
        \item there exists a matrix $\Lambda\in\R^{d\times d}$ for which we have
        \[\partial_{\theta_j}\log q_{\theta^\star}=\sum_{k=1}^{d}\Lambda_{jk}\Phi_{\theta^\star,k}, \quad j=1,\dots,d.\]
    \end{enumerate}
\end{thm}

Let us provide two examples in which the score matching estimator is asymptotically efficient.
\begin{exm}[Normal distribution]
   For the normal distribution $N(\mu,\Sigma)$, the Wasserstein score functions are given as follows \citep{Amari2024}:
   \begin{align*}
        \Phi_{\mu_j}(x)&=x_j-\mu_j,&&j=1,\dots,p,\\
        \Phi_{\Sigma_{jk}}(x)&=- \frac{1}{2}\trace(S_{jk}\Sigma)+\frac{1}{2}(x-\mu)^{\top}S_{jk}(x-\mu),&&1\leq j\leq k\leq p.
   \end{align*}
   Here, $S_{jk}\in\R^{p\times p}$ is the symmetric matrix defined by the unique solution of the Sylvester equation
   \[S_{jk}\Sigma+\Sigma S_{jk}=\begin{cases}E_{jj}&j=k\\E_{jk}+E_{kj}&j\neq k\end{cases},\]
   where $E_{jk}\in\R^{p\times p}$ is the matrix whose $(j,k)$-th entry is 1 and all other entries are 0.
   The Fisher score functions are givens as follows:
   \begin{align*}
       \partial_{\mu_j}\log q_{\theta}(x)&=\sum_{k=1}^{p}(\Sigma^{-1})_{jk}(x_k-\mu_k),\\
       \partial_{\Sigma_{jk}}\log q_{\theta}(x)&=-\trace S_{jk}+\frac{1}{2}(x-\mu)^{\top}(S_{jk}\Sigma^{-1}+\Sigma^{-1}S_{jk})(x-\mu).
   \end{align*}
   By the existence and uniqueness of the solution of the Sylvester equation, there exist unique constants $c_{lm}\in\R$ such that $S_{jk}\Sigma^{-1}+\Sigma^{-1}S_{jk}=\sum_{l\leq m}c_{lm}S_{lm}$. In addition, we have
   \[\sum_{l\leq m}c_{lm}\trace(S_{lm}\Sigma)=\trace\big((S_{jk}\Sigma^{-1}+\Sigma^{-1}S_{jk})\Sigma\big)=2\trace S_{jk}.\]
   Thus, we obtain
   \begin{align*}
       \partial_{\mu_{j}}\log q_{\theta}&=\sum_{k=1}^{p}(\Sigma^{-1})_{jk}\Phi_{\mu_{k}},\\
       \partial_{\Sigma_{jk}}\log q_{\theta}&=\sum\limits_{l\leq m}c_{lm}\Phi_{\Sigma_{lm}},
   \end{align*}
   which, together with Theorem \ref{thm:lowerbound}, concludes that the score matching is asymptotically efficient; in this example, the score matching coincides with MLE and thus is efficient \citep{Hyvarinen2005}.
\end{exm}

\begin{exm}[Generalized gamma-type distribution]
    Consider the model on $\R$ whose density is defined by
    \[q_\theta(x)=\frac{\theta^{\beta+1/2}}{\Gamma(\beta+1/2)}x^{2\beta}e^{-x^2\theta},\]
    where $\beta\in\mathbb{N}$ is a fixed shape parameter and $\theta>0$ is a scale parameter.
    Since the Fisher score function is given by $\partial_{\theta}\log q_\theta(x)=-x^{2}+(2\beta+1)/(2\theta)$, we obtain
    \[\A_{\theta}(\nabla_x\partial_{\theta}\log q_{\theta})=-4\theta\partial_\theta\log q_\theta.\]
    Thus, the Wasserstein score function is given by $\Phi_{\theta}(x)=-x^{2}/(4\theta)+(2\beta+1)/(8\theta^2)$,
    which, together with Theorem \ref{thm:lowerbound}, concludes that the score matching is asymptotically efficient. This gives a counterexample to the latter part of Theorem 13.5 in \cite{Amari2016}.
\end{exm}

\begin{remark}[Iterative subspace construction]
Our Theorem~\ref{thm:improve} shows that the score matching is improved upon by incorporating appropriate $W$-orthogonal elements.
If these $W$-orthogonal elements corresponds to the Wasserstein score functions, the resulting SMoM estimator attains the Fisher efficiency.
Even if we cannot identify such elements, the asymptotic variance approaches the efficiency bound as we incorporate more orthogonal elements.
As illustrated in Figure~\ref{fig:schematic_diagram_improvement}, this procedure is interpreted as the iterative orthonormal subspace expansion of test functions to cover the Wasserstein score functions.
\end{remark}

To prove Theorem~\ref{thm:lowerbound}, we need the two lemmas.
The following lemma connects two inner products related to $W$- and $\mathcal{A}_{\theta^{\star}}$-orthogonality.
The proof is given in \ref{sec:Wasserstein_proof}.
\begin{lem}
    \label{lem:connect_two_innerprod}
    Under the regularity conditions, for $f_\theta\colon\R^p\to\R^p$, we have
    \[\E_{\theta^\star}\left[\ang{f_{\theta^\star},\nabla_{x}\partial_{\theta_j}\log q_{\theta^\star}}\right]=\E_{\theta^\star}\left[(\A_{\theta^\star}f_{\theta^\star})(\A_{\theta^\star}(\nabla_x\Phi_{\theta^\star,j}))\right],\quad j=1,\dots,d.\]
    In particular, 
    for $j=1,\dots,d$,
    if the test function $u^\star$ is $W$-orthogonal to $\nabla_{x}\partial_{\theta_j}\log q_{\theta^\star}$,
    then it is $\A_{\theta^\star}$-orthogonal to $\nabla_x\Phi_{\theta^\star,j}$.
\end{lem}

The following lemma  shows that the gap between $\AVar[\hat\theta_\mathrm{MLE}]$ and $\AVar[\hat\theta_\mathrm{SM}]$ is characterized by the $W$-orthogonal elements $u_k^\star$.
The proof is given in \ref{sec:Wasserstein_proof}.
\begin{lem}
    \label{lem:gap_mle_sm}
    Under the regularity conditions, we have
    \[\AVar\left[\hat\theta_\mathrm{SM}\right]-\AVar\left[\hat\theta_\mathrm{MLE}\right]=G^{-1}\E_{\theta^\star}\left[
        \begin{pmatrix}\A_{\theta^\star}u_1^\star\\\vdots\\\A_{\theta^\star}u_d^\star\end{pmatrix}
        \begin{pmatrix}\A_{\theta^\star}u_1^\star\\\vdots\\\A_{\theta^\star}u_d^\star\end{pmatrix}^\top
    \right]G^{-1},\]
    where $u_j^\star$ is the $W$-orthogonal element of $\nabla_x\Phi_{\theta^\star,j}$ given by the $W$-orthogonal decomposition \eqref{eq:orthogonal_decomp_testfunction}.
\end{lem}

Lemma~\ref{lem:gap_mle_sm} tells us that if the $\A_{\theta^\star}$-orthogonal terms $\A_{\theta^\star}u_j^\star$ is dominant in $\A_{\theta^\star}(\nabla_x\Phi_{\theta^\star,j})$, the asymptotic variance $\AVar[\hat\theta_\mathrm{SM}]_{jj}$ is much worse than $\AVar[\hat\theta_\mathrm{MLE}]_{jj}$.
Conversely, if $\A_{\theta^\star}u_j^\star$ is negligible compared to $\A_{\theta^\star}(\nabla_{x}\partial_{\theta_j}\log q_{\theta^\star})$, we have $\AVar[\hat\theta_\mathrm{SM}]_{jj}\approx\AVar[\hat\theta_\mathrm{MLE}]_{jj}$.

\begin{exm}[Generalized normal distribution]
Consider the special case of the generalized normal distribution whose density is defined by
\begin{equation}
    \label{eq:generalized_normal}
    q_{\theta}(x)=\frac{\beta\theta^{\frac{1}{2\beta}}}{\Gamma(1/(2\beta))}\exp\left(-\theta x^{2\beta}\right),
\end{equation}
where $x\in\R$, $\beta\in\mathbb{N}$ is a fixed shape parameter, and $\theta>0$ is the scale parameter.
If $\beta=1$, the distribution is $N(0,1/(2\theta))$.
Both of the score matching estimator and maximum likelihood estimator have the closed-form solutions:
\[
\hat\theta_\mathrm{MLE}=\frac{n}{2\beta\sum_{i=1}^{n}X_i^{2\beta}},\quad
\hat\theta_\mathrm{SM}=\frac{2\beta-1}{2\beta}\frac{\sum_{i=1}^{n} X_i^{2\beta-2}}{\sum_{i=1}^{n} X_i^{4\beta-2}}.
\]
They coincide only in case $\beta=1$.
The SMoM estimator based on the test function $f\colon\R\to\R$ also has the closed form:
\[
\hat\theta_\mathrm{SMoM}=\frac{\sum_{i=1}^{n}\frac{\mathrm{d}}{\mathrm{d} x}f(X_i)}{2\beta\sum_{i=1}^{n}X_{i}^{2\beta-1}f(X_i)}.
\]
The MLE corresponds to the case of $f(x)=x$, and the score matching estimator corresponds to the case of $f(x)=-2\beta x^{2\beta-1}$.
The Wasserstein score function $\Phi_\theta$ has the explicit form:
\[\Phi_\theta(x)=-\frac{1}{4\beta\theta}x^2+\frac{\Gamma(3/(2\beta))}{4\beta\theta^{1+1/\beta}\Gamma(1/(2\beta))}.\]
Observe that $\nabla_x\Phi_{\theta}$ is orthogonally decomposed as follows:
\[
\nabla_x\Phi_{\theta}=-\frac{\Gamma(1/(2\beta))}{4\beta^2(2\beta-1)\theta^{1/\beta}\Gamma(1-1/(2\beta)}\left(-2\beta x^{2\beta-1}+u_{\theta}\right),
\]
where $u_{\theta}$ is defined by
\[
u_{\theta}(x)\coloneqq-\frac{2\beta(2\beta-1)\theta^{-1+1/\beta}\Gamma(1-1/(2\beta))}{\Gamma(1/(2\beta))}x+2\beta x^{2\beta-1}
\]
and satisfies $\E_{\theta}[\ang{-2\beta X^{2\beta-1},u_{\theta}}]=0$.
Combine it with Lemma~\ref{lem:gap_mle_sm} to get  the asymptotic relative efficiency:
\begin{equation}
    \label{eq:generalized_normal_ratio}
        \frac{\AVar\left[\hat\theta_\mathrm{MLE}\right]}{\AVar\left[\hat\theta_\mathrm{SM}\right]}
        =\frac{2\beta-1}{2(3\beta-2)}\frac{\Gamma(1-1/(2\beta))^2}{\Gamma(1+1/(2\beta))\Gamma(2-3/(2\beta))}
        \to\frac13\quad(\beta\to\infty).
\end{equation}
Figure~\ref{fig:generalized_normal_ratio} shows the behavior of \eqref{eq:generalized_normal_ratio} with respect to $\beta$.
It tells us that the gap in the asymptotic variance between the score matching estimator and the MLE widens as $\beta$ increases.
\end{exm}

\begin{figure}[H]
    \centering
    \includegraphics[width=0.75\linewidth]{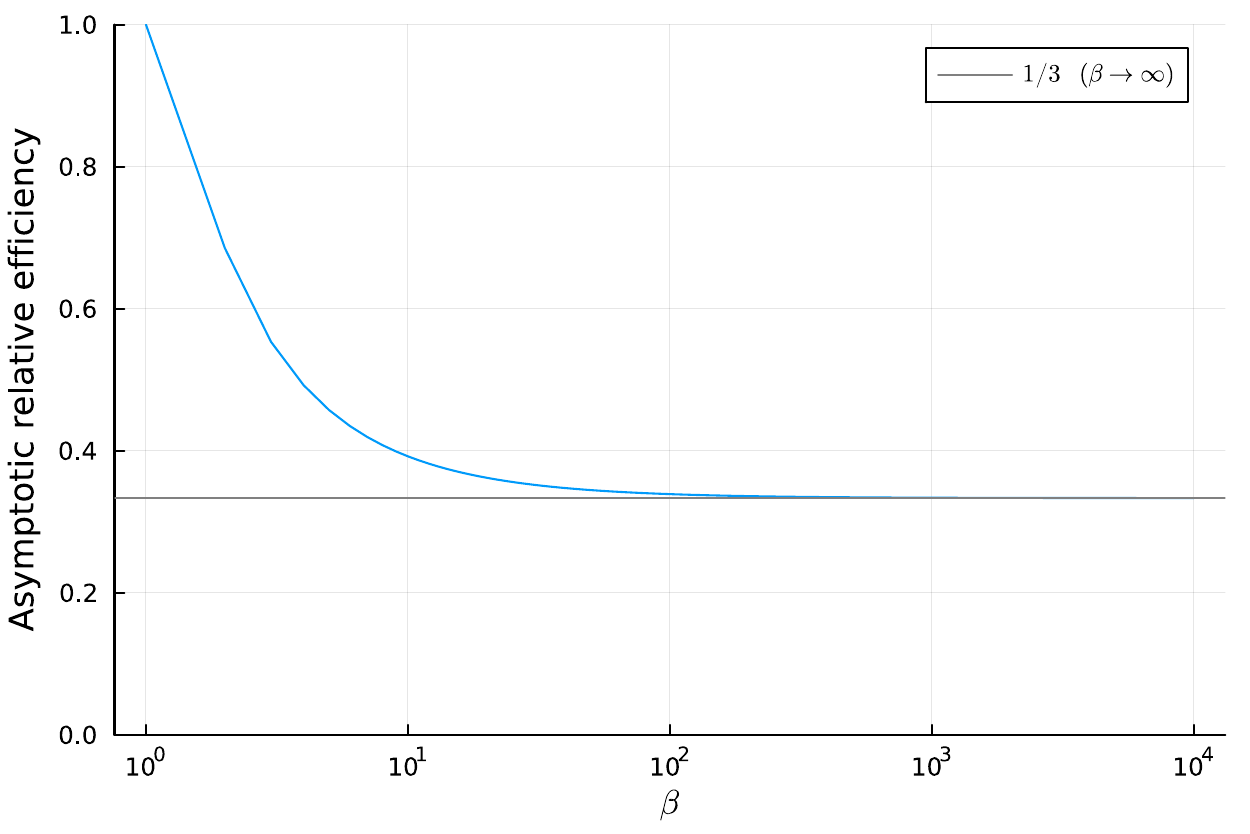}
    \caption{The asymptotic relative efficiency $\AVar[\hat\theta_\mathrm{MLE}]/\AVar[\hat\theta_\mathrm{SM}]$ with respect to $\beta$ (blue) along with its limits (gray).}
    \label{fig:generalized_normal_ratio}
\end{figure}

\section{Numerical experiments}
\label{sec:simu_appl}
In this section, we provide numerical experiments illustrating the SMoM estimator constructed in Theorem~\ref{thm:improve}; see also \ref{sec:simu_appl_M} for additional experiments.
We focus on estimating the parameters $\theta$.
The estimates $\hat\theta_\mathrm{SM},\hat\theta[\theta^\star]$, and $\hat\theta[\hat\theta_\mathrm{SM}]$ are calculated using an i.i.d. sample of size $n$ drawn from the distribution with $\theta^\star$.
This procedure is iterated in 1000 times, and the MSE of each estimates are calculated.
The estimates of asymptotic relative efficiency given by \eqref{eq:estimate_improvement} or \eqref{eq:estimate_improvement_M} are also calculated.
$\tilde{v}_1,\dots,\tilde{v}_K$ is constructed using neural networks.
Specifically, $\tilde{v}_\alpha$ is implemented as a neural network with five hidden layers, each consisting of three nodes, and $p$-dimensional input and output layers.
We employ $\tanh$ as the activation function, and all parameters are randomly initialized with values drawn from $N(0,1)$.
Since the estimates $\hat\theta[\theta^\star],\hat\theta[\hat\theta_\mathrm{SM}]$ depend on the choice of $\tilde{v}_1,\dots,\tilde{v}_K$, we evaluate the performance over 10 different pairs of initializations.
The expectations are approximated via Monte Carlo integration with 1000 samples.

\subsection{Generalized normal distribution}
We focus on estimating the scale parameter of $\theta$ of the generalized normal distribution whose density given by \eqref{eq:generalized_normal}.
Here, $\theta^\star$ is set to $\theta^\star=\Gamma(3/(2\beta))^{2\beta}/\Gamma(1/(2\beta))^{2\beta}$, which ensures that the distribution has unit variance.
$\beta$ is set to 2.

The sample size $n$ varies in $\{10,100,1000\}$. 
The number of $W$-orthogonal elements $K$ varies in $\{1,2,4,8\}$.
We also calculate $\hat\theta_\mathrm{MLE}$ as a benchmark.

Table~\ref{tab:gnormal_improvement_beta2} shows that both of $\hat\theta[\theta^\star],\hat\theta[\hat\theta_\mathrm{SM}]$ typically have lower variance than the score matching estimator, whereas this improvement does not hold for small $K$ or small $n$.
Both of them are comparable to the result of the MLE under some $\tilde{v}_\alpha$.
Figure~\ref{fig:gnormal_testfunction_beta2} shows that the test functions of $\hat\theta[\theta^\star],\hat\theta[\hat\theta_\mathrm{SM}]$ are almost same as that of the score matching estimator with $K=1$, and close to that of the MLE with $K=8$.

\begin{table}[H]
    \centering
    \caption{MSE ratio for $\hat\theta[\hat\theta_\mathrm{SM}]$, $\hat\theta[\theta^\star]$, and $\hat\theta_\mathrm{MLE}$ relative to $\hat\theta_\mathrm{SM}$. 
    For $\hat\theta[\hat\theta_\mathrm{SM}]$ and $\hat\theta[\theta^\star]$, the results are summarized as median (min, max) over different pairs of $\tilde{v}_\alpha$. A value smaller than 1 indicates better performance. For each sample size $n$, the median corresponding to the best-performing $K$ is underlined. Values exceeding 100 are omitted and represented by a hyphen (-).}
    \begin{tabular}{lllll}
  \hline
   && $\hat\theta[\hat\theta_\mathrm{SM}]$ & $\hat\theta[\theta^\star]$ & $\hat\theta_\mathrm{MLE}$\\\hline
  $n=10$ & $K=1$ & $1.015 \,(0.835,1.347)$ & $2.764 \,(0.683,-)$ & $0.536$ \\
   & $K=2$ & $0.858 \,(0.635,1.075)$ & $1.355 \,(0.539,5.118)$ & \\
   & $K=4$ & $\underline{0.646} \,(0.408,0.964)$ & $0.635 \,(0.347,2.688)$ & \\
   & $K=8$ & $0.740 \,(0.605,0.862)$ & $\underline{0.565} \,(0.447,1.022)$ & \\\hline
  $n=100$ & $K=1$ & $1.020 \,(0.944,1.054)$ & $1.022 \,(0.946,1.083)$ & $0.660$ \\
   & $K=2$ & $1.021 \,(0.735,1.065)$ & $1.015 \,(0.706,1.121)$ & \\
   & $K=4$ & $\underline{0.748} \,(0.661,0.851)$ & $0.742 \,(0.660,0.825)$ & \\
   & $K=8$ & $0.791 \,(0.695,1.156)$ & $\underline{0.736} \,(0.689,0.772)$ & \\\hline
  $n=1000$ & $K=1$ & $1.031 \,(0.984,1.042)$ & $1.006 \,(0.968,1.041)$ & $0.685$ \\
   & $K=2$ & $0.902 \,(0.665,1.050)$ & $0.899 \,(0.702,1.061)$ & \\
   & $K=4$ & $\underline{0.750} \,(0.700,1.019)$ & $\underline{0.750} \,(0.694,1.023)$ & \\
   & $K=8$ & $0.831 \,(0.698,2.758)$ & $0.829 \,(0.692,2.953)$ & \\\hline
\end{tabular}

    \label{tab:gnormal_improvement_beta2}
\end{table}

\begin{figure}[H]
    \centering
    \includegraphics[width=\linewidth]{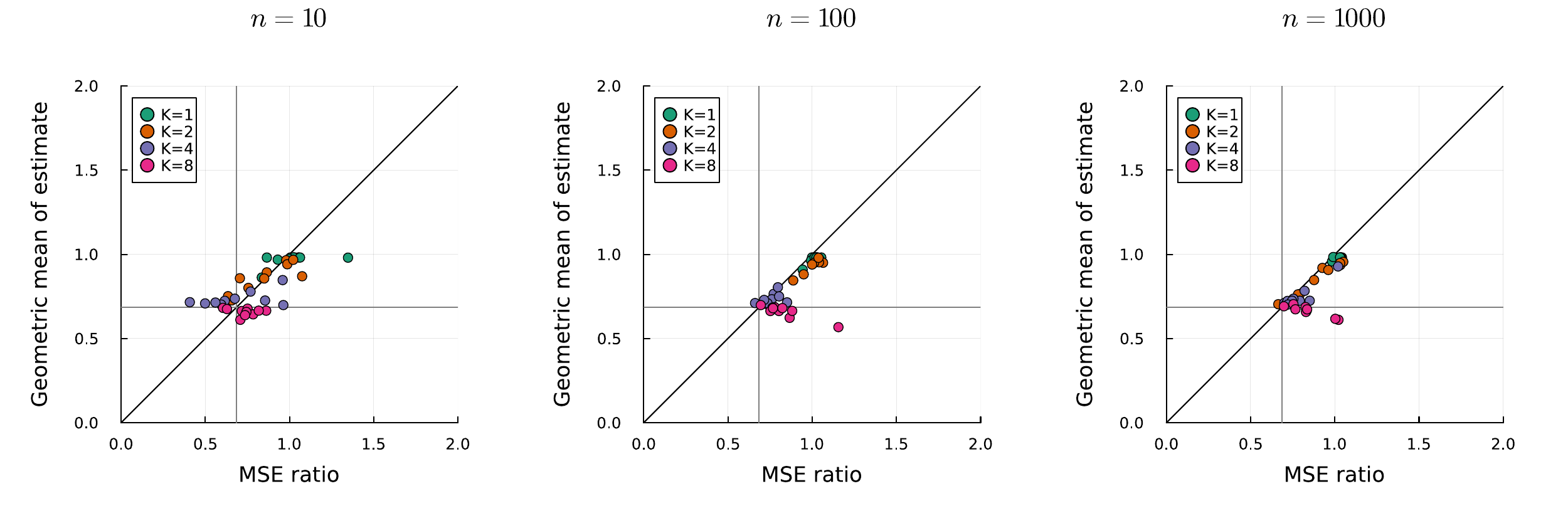}
    \caption{MSE ratio for $\hat\theta[\hat\theta_\mathrm{SM}]$ relative to $\hat\theta_\mathrm{SM}$ versus the geometric mean of estimates given by \eqref{eq:estimate_improvement}. Points of the same color correspond to different pairs of $\tilde{v}_\alpha$. The horizontal and vertical line (gray) represent the asymptotic relative efficiency for MLE calculated by \eqref{eq:generalized_normal_ratio}. The horizontal axis is truncated at 2.
    Values less than 1 on the horizontal axis indicate that corresponding SMoM estimator improves the variance of the score matching estimator. Points near the diagonal indicate that the estimate of the asymptotic relative efficiency is reliable.}
    \label{fig:gnormal_improvement_estimate_beta2}
\end{figure}

\begin{figure}[H]
    \centering
    \includegraphics[width=\linewidth]{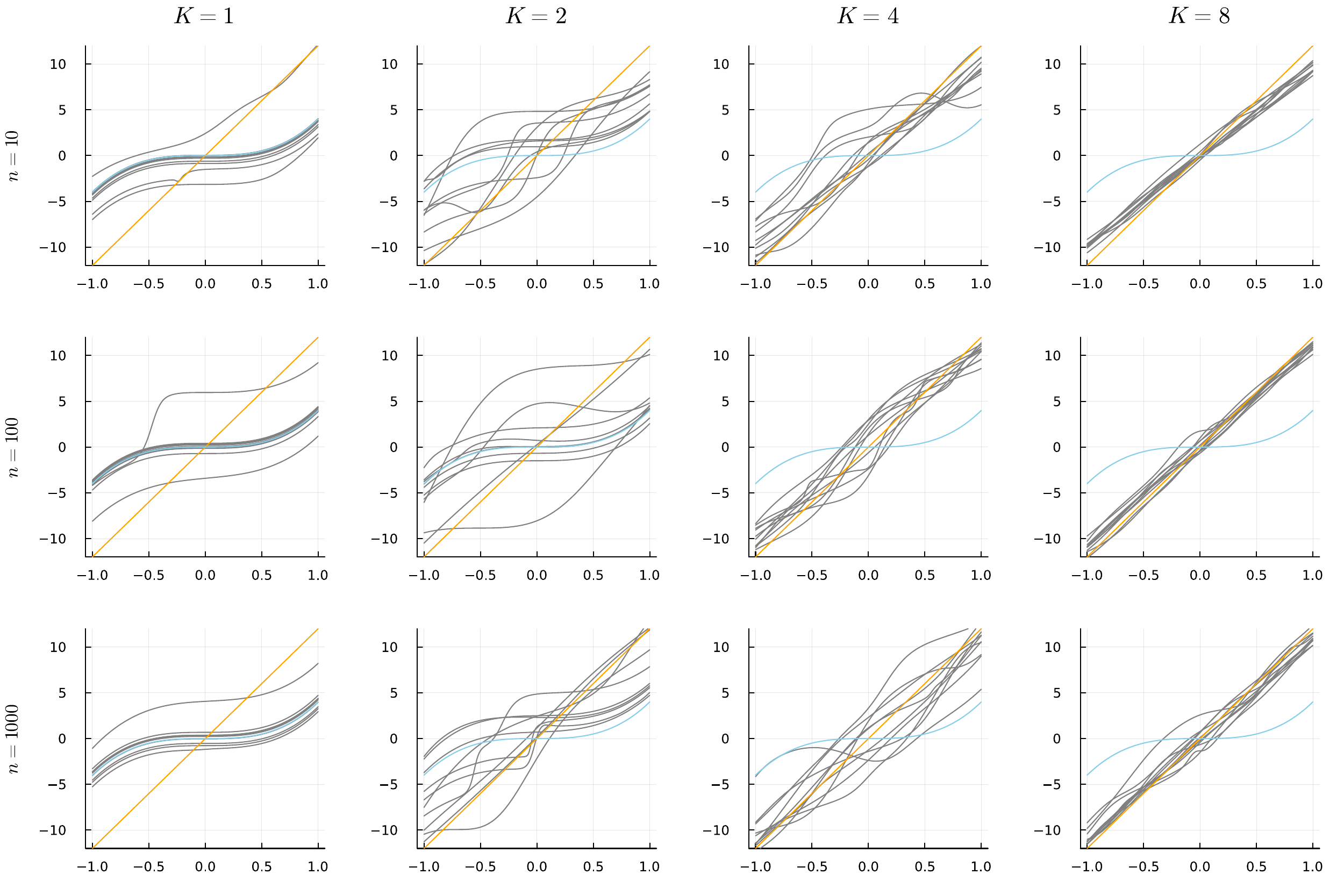}
    \caption{Test functions for $\hat\theta_\mathrm{SM}$ (blue), $\hat\theta_\mathrm{MLE}$ (orange), and $\hat\theta[\hat\theta_\mathrm{SM}]$ (gray). For $\hat\theta[\hat\theta_\mathrm{SM}]$, the mean value over iterations are plotted, where each line corresponds to a different pairs of $\tilde{v}_\alpha$. A test function close to the test function of the MLE implies that the corresponding SMoM estimator is also close to the MLE.}
    \label{fig:gnormal_testfunction_beta2}
\end{figure}

\subsection{Polynomially tilted pairwise interaction model}
Consider the $(p-1)$-dimensional unit sphere $\mathbb{S}^{p-1}\coloneqq\{x\in\mathbb{R}^{p}\mid\norm{x}=1\}$,
Let $\mathbb{S}^{p-1}_+$ be its non-negative orthant, i.e., $\mathbb{S}^{p-1}_+\coloneqq\mathbb{S}^{p-1}\cap[0,\infty)^p$.
The polynomially tilted pairwise interaction (PPI) model \citep{Scealy2023} has the following density:
\[q_{\theta}(x)=\frac{1}{Z(A,\mu,\beta)}\left(\prod_{j=1}^{p}x_{j}^{1+2\beta_j}\right)\exp\left(x^{2\top}Ax^2+\mu^\top x^2\right),\quad x\in\mathbb{S}^{p-1}_+,\]
where $\theta=(A,\mu)$, $A\in\R^{p\times p}$ is a symmetric matrix, $\mu\in\R^p$, and $\beta\in(-1,\infty)^p$ is a fixed shape parameter.
For identifiability, we fix $A_{pj}=A_{jp}=0$ for $j=1,\dots,d$ and $\mu_p=0$.
This density is derived by the square root transformation from the unit simplex to $\mathbb{S}^{p-1}_+$. See \citet{Scealy2023} for details. 
Also, about the weighted score matching estimator $\hat\theta_\mathrm{wSM}$ and corresponding SMoM estimators, see \ref{sec:main_result_M}.

The sample size $n$ is set to 100.
We set to $p=3$, $A^\star=\operatorname{diag}(1,1,0)$, $\mu^\star=(0,0,0)^\top$ and $\beta=(-0.5,-0.5,-0.5)^\top$.
We use the weight function $w(x)=\prod_{j=1}^p x_j$.
The number of $W$-orthogonal elements $K$ varies in $\{3,6,12,24\}$.
For calculation of $\tilde{v}_\alpha$, to ensure the output is a valid vector field on $\S^{p-1}_+$, the output vector is projected onto the tangent space $T_x\S^{p-1}_+$ via the orthogonal projection $\mathrm{P}_x=I_p-xx^\top$.

Table~\ref{tab:ppi_improvement} shows that both of $\hat\theta[\theta^\star],\hat\theta[\hat\theta_\mathrm{wSM}]$ improve the variance of the score matching estimator for $K\leq12$.
$\hat\theta[\hat\theta_\mathrm{wSM}]$ typically shows better performance than $\hat\theta[\theta^\star]$.
While the efficiency initially improves as $K$ increases, it begins to deteriorate at $K=24$, potentially due to numerical instability.
Figure~\ref{fig:ppi_improvement_estimate} shows that the estimate of asymptotic relative efficiency is typically well-performed for $K\leq12$.
These results confirm our theory works well in finite-sample settings, although we need to choose an appropriate $K$.

\begin{table}[H]
    \centering
    \caption{MSE ratio of PPI model for $\hat\theta[\hat\theta_\mathrm{wSM}]$, $\hat\theta[\theta^\star]$ relative to $\hat\theta_\mathrm{wSM}$. The results are represented in the format: median (min, max) across different pairs of $\tilde{v}_\alpha$. A value smaller than 1 indicates better performance. The best-performing $K$ is underlined. Values exceeding 100 are omitted and represented by a hyphen (-).}
    \begin{tabular}{lcll}
  \hline
  && $\hat\theta[\hat\theta_\mathrm{wSM}]$ & $\hat\theta[\theta^\star]$ \\
  \hline
  $K=3$ & $A_{11}$ & $0.919 \,(0.883,0.967)$ & $0.937 \,(0.912,0.961)$ \\
   & $A_{22}$ & $0.913 \,(0.840,0.952)$ & $0.929 \,(0.853,0.967)$ \\
   & $A_{12}$ & $0.909 \,(0.866,0.972)$ & $0.923 \,(0.893,0.973)$ \\
   & $\mu_1$ & $0.932 \,(0.909,0.982)$ & $0.948 \,(0.929,0.985)$ \\
   & $\mu_2$ & $0.919 \,(0.880,0.959)$ & $0.933 \,(0.892,0.968)$ \\\hline
  $K=6$ & $A_{11}$ & $0.810 \,(0.695,0.913)$ & $0.829 \,(0.704,0.907)$ \\
   & $A_{22}$ & $0.792 \,(0.680,0.884)$ & $0.812 \,(0.710,0.908)$ \\
   & $A_{12}$ & $0.830 \,(0.652,0.888)$ & $0.846 \,(0.657,0.899)$ \\
   & $\mu_1$ & $0.841 \,(0.640,0.907)$ & $0.850 \,(0.654,0.905)$ \\
   & $\mu_2$ & $0.788 \,(0.670,0.911)$ & $0.805 \,(0.679,0.917)$ \\\hline
  $\underline{K=12}$ & $A_{11}$ & $0.694 \,(0.648,0.765)$ & $0.732 \,(0.658,0.878)$ \\
   & $A_{22}$ & $0.664 \,(0.561,0.806)$ & $0.683 \,(0.595,0.851)$ \\
   & $A_{12}$ & $0.702 \,(0.622,0.825)$ & $0.709 \,(0.624,0.882)$ \\
   & $\mu_1$ & $0.706 \,(0.615,0.774)$ & $0.736 \,(0.613,0.830)$ \\
   & $\mu_2$ & $0.684 \,(0.559,0.786)$ & $0.694 \,(0.596,0.836)$ \\\hline
  $K=24$ & $A_{11}$ & $1.886 \,(0.638,-)$ & $33.035 \,(0.636,-)$ \\
   & $A_{22}$ & $2.054 \,(0.632,66.790)$ & $81.787 \,(0.653,-)$ \\
   & $A_{12}$ & $2.676 \,(0.613,-)$ & $40.986 \,(0.618,-)$ \\
   & $\mu_1$ & $1.985 \,(0.603,-)$ & $24.774 \,(0.599,-)$ \\
   & $\mu_2$ & $2.065 \,(0.622,-)$ & $55.931 \,(0.624,-)$ \\\hline
\end{tabular}

    \label{tab:ppi_improvement}
\end{table}

\begin{figure}[H]
    \centering
    \includegraphics[width=\linewidth]{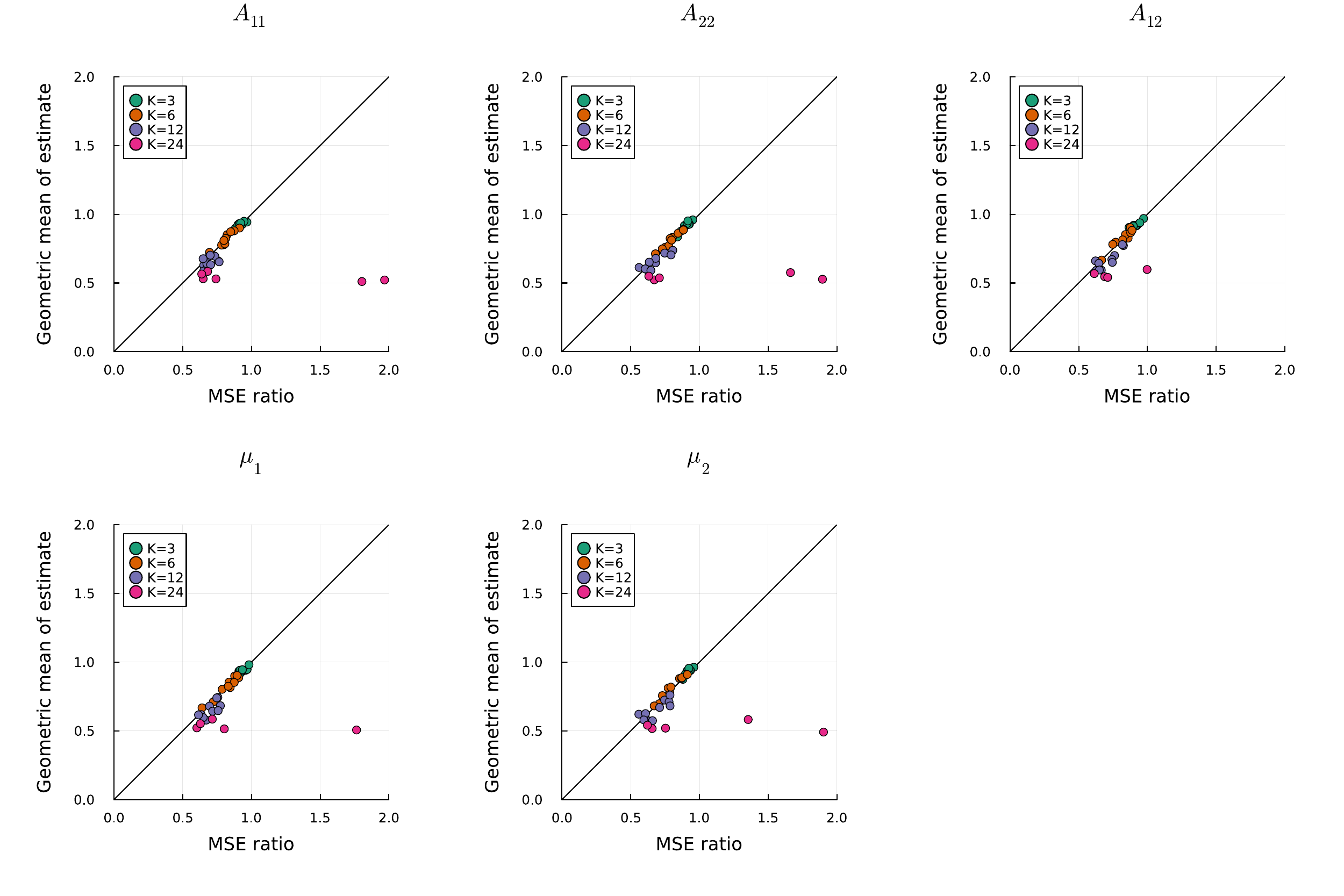}
    \caption{MSE ratio of PPI model for $\hat\theta[\hat\theta_\mathrm{wSM}]$ relative to $\hat\theta_\mathrm{wSM}$ versus the geometric mean of estimates given by \eqref{eq:estimate_improvement_M}. Points of the same color correspond to different pairs of $\tilde{v}_\alpha$. The horizontal axis is truncated at 2.
    Values less than 1 on the horizontal axis indicate that corresponding SMoM estimator improves the variance of the score matching estimator. Points near the diagonal indicate that the estimate of the asymptotic relative efficiency is reliable.}
    \label{fig:ppi_improvement_estimate}
\end{figure}

\section{Conclusion}
\label{sec:concl}
In this paper, we studied the geometry of SMoM estimators focusing on $W$-orthogonality and $\A_{\theta^\star}$-orthogonality.

The canonical decomposition (Theorem~\ref{thm:smom_decomp}) gives the asymptotoic linear representation of SMoM estimators after centering the score matching estimator, and it implies that the asymptotic linear representation of SMoM estimators is characterized by its $W$-orthogonal term.
Using the fact that the $W$-orthogonality does not imply $\A_{\theta^\star}$-orthogonality, we constructed a SMoM estimator which improves the asymptotic variance of the score matching estimator (Theorem~\ref{thm:improve}).
Through the numerical experiments, we comfirmed that the improvement is effective in finite sample settings under an appropriate choice of $K$.

The unexpected connection between the Wasserstein geometry and SMoM presents further characterization of the asymptotic efficiency of the score matching estimator.
We showed that the score matching estimator is asymptotically efficient if and only if the Fisher score functions span the same space as the Wasserstein score functions (Theorem~\ref{thm:lowerbound}).
We also showed that the $W$-orthogonal term of the Wasserstein score function characterizes the gap of the asymptotic variance between the score matching estimator and the MLE (Lemma~\ref{lem:gap_mle_sm}).
The geometry of SMoM may provide further connection between the three geometries: the geometry of score matching, Fisher-Rao geometry, and Wasserstein geometry.

\section{Acknowledgement}

The authors would like to thank Takeru Matsuda, Yoshikazu Terada, and Shotaro Yagishita for helpful comments.
This work is supported by JSPS KAKENHI (21H05205, 23K11024), and MEXT (JPJ010217).

\bibliographystyle{biometrika}
\bibliography{cite.bib}

\appendix
\renewcommand{\thesection}{Appendix \Alph{section}}

\section{Regularity conditions}
\label{sec:regularity}

This appendix summarizes the regularity conditions that we use in this paper.

We assume the following conditions for the model $\{q_\theta\mid\theta\in\Theta\}$ and test functions $f_\theta$:

\begin{enumerate}
    \item[A1.] $\E_{\theta}[(\A_{\theta}f_\theta)^2]<\infty$ and $\E_{\theta}[\A_{\theta}f_{\theta}]=0$.
    \item[A2.] $\E_{\theta}[\A_{\theta}(\partial_{\theta_j}f_{\theta})]=0,\quad j=1,\dots,d$.
    \item[A3.] $\E_{\theta}[|\ang{f_\theta,\nabla_x\partial_{\theta_j}\log q_{\theta}}|]<\infty,\quad j=1,\dots,d$.
    \item[A4.] $\partial_{\theta_j}\E_{\theta}[\A_{\theta}f_\theta]=\int\partial_{\theta_j}(q_\theta\A_{\theta}f_{\theta})\d x,\quad j=1,\dots,d$.  
\end{enumerate}
The condition $\E_{\theta}[\A_{\theta}f_\theta]=0$ is satisfied if the boundary condition $\lim_{\|x\|\to\infty}q_{\theta}(x)f_{\theta}(x)=0$ holds.
Condition A2 is used in the proof of Lemma~\ref{lem:G_mat}, the latter part of Theorem~\ref{thm:improve}, and Lemma~\ref{lem:connect_two_innerprod}.
Condition A4 is used in the proof of Lemma~\ref{lem:connect_two_innerprod}.

We also assume the following conditions for the Fisher score functions $\partial_{\theta_j}\log q_{\theta}$ and the Wasserstein score functions $\Phi_{\theta,j}$:

\begin{enumerate}
    \item[A5.] $\E_{\theta}[\A_{\theta}(\Phi_{\theta,j}\nabla_x\partial_{\theta_k}\log q_{\theta})]=0$.
    \item[A6.] $\E_{\theta}[\A_{\theta}(\partial_{\theta_j}\log q_{\theta}\nabla_x\Phi_{\theta,k})]=0$.
    \item[A7.] $\E_{\theta}[\A_{\theta}(\partial_{\theta_j}\log q_{\theta}\nabla_x\partial_{\theta_k}\log q_\theta)]=0$.
    \item[A8.] $\E_{\theta}[\A_{\theta}(\Phi_{\theta,j}\nabla_x\Phi_{\theta,k})]=0$. 
\end{enumerate}
These conditions are used in the proof of Theorem~\ref{thm:lowerbound}.

For the SMoM estimator based on the test functions $f_{\theta,1},\dots,f_{\theta,d}$, we assume the following conditions:
\begin{enumerate}
    \item[A9.] $\A_{\theta^\star}f_{\theta^\star,1},\dots,\A_{\theta^\star}f_{\theta^\star,d}$ are linearly independent.
    \item[A10.] The matrix $G_\mathrm{SMoM}\in\R^{d\times d}$ whose $(j,k)$-th entry is defined by \[(G_\mathrm{SMoM})_{jk}\coloneqq\E_{\theta^\star}\Big[\partial_{\theta_k}\A_{\theta}f_{\theta,j}\Big|_{\theta=\theta^\star}\Big]\] is invertible.
    \item[A11.] $\hat\theta_\mathrm{SMoM}$ is consistent and asymptotic linear; that is, $\hat\theta_\mathrm{SMoM}-\theta^\star$ has the following representation: \[\hat\theta_\mathrm{SMoM}-\theta^\star=-G_\mathrm{SMoM}^{-1}\frac{1}{n}\sum_{i=1}^{n}\begin{pmatrix}\A_{\theta^\star}f_{\theta^\star,1}(X_i)\\\vdots\\\A_{\theta^\star}f_{\theta^\star,d}(X_i)\end{pmatrix} + \smallop.\]
\end{enumerate}

\section{Proofs in Section~\ref{sec:Wasserstein}}
\label{sec:Wasserstein_proof}

This appendix provides the proofs for the results in Section~\ref{sec:Wasserstein}.

\begin{proof}[Proof of Lemma~\ref{lem:connect_two_innerprod}]
    By the Leibniz rule and the definition of the Wasserstein score function, we have
    \begin{align*}
        0&=\partial_{\theta_j}\E_{\theta}[\A_{\theta}f_{\theta}]|_{\theta=\theta^\star}\\
        &=\E_{\theta^\star}[\partial_{\theta_j}\A_{\theta}f_{\theta}|_{\theta=\theta^\star}]+\E_{\theta^\star}[(\A_{\theta^\star}f_{\theta^\star})(\partial_{\theta_j}\log{q}_{\theta^\star})]\\
        &=\E_{\theta^\star}[\partial_{\theta_j}\A_{\theta}f_{\theta}|_{\theta=\theta^\star}]-\E_{\theta^\star}[(\A_{\theta^\star}f_{\theta^\star})\A_{\theta^\star}(\nabla_x\Phi_{\theta^\star,j})],\quad j=1,\dots,d.
    \end{align*}
    Combining this with the relation $\E_{\theta^\star}[\partial_{\theta_j}\A_{\theta}f_{\theta}|_{\theta=\theta^\star}]=\E_{\theta^\star}[\ang{f_{\theta^\star},\nabla_{x}\partial_{\theta_j}\log q_{\theta^\star}}]$ from Lemma~\ref{lem:G_mat}, we obtain
    \[\E_{\theta^\star}[\ang{f_{\theta^\star},\nabla_{x}\partial_{\theta_j}\log q_{\theta^\star}}]=\E_{\theta^\star}[(\A_{\theta^\star}f_{\theta^\star})(\A_{\theta^\star}\nabla_x\Phi_{\theta^\star,j})],\]
    which completes the proof.
\end{proof}

\begin{proof}[Proof of Lemma~\ref{lem:gap_mle_sm}]
    Let $F\in\R^{d\times d}$ be the inner product matrices whose $(j,k)$-th entry is defined by
    \[F_{jk}\coloneqq \E_{\theta^\star}[\ang{\nabla_x\Phi_{\theta^\star,j},\nabla_{x}\partial_{\theta_k}\log q_{\theta^\star}}],\]
    and $u_j^\star$ is defined by
    \begin{equation}
        \label{eq:wasserstein_score_orth}
        u_{j}^{\star}\coloneqq\sum\limits_{k=1}^{d}(GF^{-1})_{jk}\nabla_x\Phi_{\theta^\star,k}-\nabla_{x}\partial_{\theta_j}\log q_{\theta^\star}.
    \end{equation}
    Note that this $u_j^\star$ is $W$-orthogonal to the test functions of score matching:
    \[\E_{\theta^\star}[\ang{u_j^\star,\nabla_x\partial_{\theta_k}\log q_{\theta^\star}}]=0,\quad k=1,\dots,d.\]
    We will show the relation 
    \begin{align}\AVar\left[\hat\theta_\mathrm{SM}\right]-\AVar\left[\hat\theta_\mathrm{MLE}\right]=G^{-1}TG^{-1},
    \label{eq:target in Lemma 4}
    \end{align}
    where $T\in\R^{d\times d}$ is the inner product matrix whose $(j,k)$-th entry is defined by 
    \[T_{jk}\coloneqq\E_{\theta^\star}[(\A_{\theta^\star}u_j^\star)(\A_{\theta^\star}u_k^\star)].\]

    We first relate $F$ to the Fisher information matrix $I(\theta^{\star})$.
    Using $u_j^\star$, we decompose $\nabla_x\Phi_{\theta^\star,j}$ as follows:
    \[\nabla_x\Phi_{\theta^\star,j}=\sum_{k=1}^d(FG^{-1})_{jk}\left(\nabla_{x}\partial_{\theta_k}\log q_{\theta^\star}+u_k^\star\right).\]
    Observe the relation
    \[\partial_{\theta_k}\partial_{\theta_j}\log q_{\theta^\star}=-\partial_{\theta_k}\A_{\theta^\star}(\nabla_x\Phi_{\theta^\star,j})=-\A_{\theta^\star}(\nabla_x\partial_{\theta_k}\Phi_{\theta^\star,j})-\ang{\nabla_x\Phi_{\theta,j},\nabla_{x}\partial_{\theta_k}\log q_{\theta^\star}}.\]
    Together with $\E_{\theta^\star}[\A_{\theta^\star}(\nabla_x\partial_{\theta_k}\Phi_{\theta^\star,j})]=0$ under the regularity conditions, taking the expectation of this yields 
    \begin{align}
    I(\theta^\star)=F.
    \label{eq:Fisher=F}
    \end{align}
    
    We next show \eqref{eq:target in Lemma 4}.
    Using Theorem~\ref{thm:smom_decomp}, we can decompose the asymptotic variance $\AVar[\hat\theta_\mathrm{MLE}]$:
    \[\AVar\left[\hat\theta_\mathrm{MLE}\right]=\AVar\left[\hat\theta_\mathrm{SM}\right]+G^{-1}(T+S+S^\top)G^{-1},\]
    where $S\in\R^{d\times d}$ is the inner produce matrix whose $(j,k)$-entry is defined by \[S_{jk}\coloneqq\E_{\theta^\star}[\A_{\theta^\star}u_{j}^\star\A_{\theta^\star}(\nabla_x\partial_{\theta_k}\log q_{\theta^\star})].\]
    Note that we have $\E_{\theta^\star}[\A_{\theta^\star}(\nabla_x\Phi_{\theta^\star,j})\A_{\theta^\star}(\nabla_{x}\partial_{\theta_k}\log q_{\theta^\star})]=G_{jk}$ by Lemma~\ref{lem:connect_two_innerprod}.
    Note also that
    substituting \eqref{eq:wasserstein_score_orth} to $S_{jk}$, we have
    \begin{align*}
        S_{jk}
        &=\sum\limits_{l=1}^{d}(GF^{-1})_{jl}\E_{\theta^\star}[\A_{\theta^\star}(\nabla_x\Phi_{\theta^\star,l})\A_{\theta^\star}(\nabla_{x}\partial_{\theta_k}\log q_{\theta^\star})]\\
        &\hphantom{=}-\E_{\theta^\star}[\A_{\theta^\star}(\nabla_{x}\partial_{\theta_j}\log q_{\theta^\star})\A_{\theta^\star}(\nabla_{x}\partial_{\theta_k}\log q_{\theta^\star})]\\
        &=(GI(\theta^\star)^{-1}G)_{jk}-\E_{\theta^\star}[\A_{\theta^\star}(\nabla_{x}\partial_{\theta_j}\log q_{\theta^\star})\A_{\theta^\star}(\nabla_{x}\partial_{\theta_k}\log q_{\theta^\star})],
    \end{align*}
    where the identity \eqref{eq:Fisher=F} is also used.
    Thus, the asymptotic variance $\AVar[\hat\theta_\mathrm{MLE}]$ is rearranged as 
    \[\AVar[\hat\theta_\mathrm{MLE}]=\AVar[\hat\theta_\mathrm{SM}]+G^{-1}TG^{-1}+2I(\theta^\star)^{-1}-2\AVar[\hat\theta_\mathrm{SM}].\]
    Since $\AVar[\hat\theta_\mathrm{MLE}]=I(\theta^\star)^{-1}$, we finally get $\AVar[\hat\theta_\mathrm{SM}]-\AVar[\hat\theta_\mathrm{MLE}]=G^{-1}TG^{-1}$,
    which completes the proof.
\end{proof}

For the proof of Theorem~\ref{thm:lowerbound}, we prepare the following lemma.
\begin{lem}
    \label{lem:lowerbound}
    For $h:\R^p\to\R$ satisfying $\E_{\theta^\star}[\A_{\theta^\star}(h\nabla_x h)]=0$, we have $\A_{\theta^\star}(\nabla_x h)=0$ if and only if the identity $\nabla_x h=0$ holds with probability 1.
\end{lem}
\begin{proof}
    If $\nabla_x h=0$, we have $\A_{\theta^\star}(\nabla_x h)=0$ by the linearity of $\A_{\theta^\star}$.
    Suppose $\A_{\theta^\star}(\nabla_x h)=0$. Observe the relation
    \[\A_{\theta^\star}(h\nabla_x h)=h\A_{\theta^\star}(\nabla_x h)+\norm{\nabla_x h}^2=\norm{\nabla_x h}^2.\]
    Taking the expectation of this yields $\E_{\theta^\star}[\norm{\nabla_x h}^2]=0$ and thus implies $\nabla_x h=0$ with probability 1.
\end{proof}

\begin{proof}[Proof of Theorem~\ref{thm:lowerbound}]
    $(1)\Rightarrow(3)$:
        Suppose the score matching estimator is asymptotically efficient: $\AVar[\hat\theta_\mathrm{SM}]=\AVar[\hat\theta_\mathrm{MLE}]$.
        By Lemma~\ref{lem:gap_mle_sm}, it implies $\E_{\theta^\star}[(\A_{\theta^\star}u_j^\star)^2]=0$, that is, $\A_{\theta^\star}u_j^\star=0$ with probability 1, where $u_j^\star$ is defined by \eqref{eq:wasserstein_score_orth}.
        Since there exists $h_j\colon\R^d\to\R$ such that 
        two identities
        $u_j^\star=\nabla_x h_j$ and $\E_{\theta^\star}[h_j\nabla_x h_j]=0$ hold under the regularity conditions, we have $u_j^\star=0$ with probability 1 by Lemma~\ref{lem:lowerbound}.
        Substituting $u_j^\star=0$ to \eqref{eq:wasserstein_score_orth}, we get
        \[\nabla_x\partial_{\theta_j}\log q_{\theta^\star}=\sum\limits_{k=1}^{d}(GF^{-1})_{jk}\nabla_x\Phi_{\theta^\star,k}.\]
        Thus, there exists a constant $c_j\in\R$ such that
        \[\partial_{\theta_j}\log q_{\theta^\star}=\sum\limits_{k=1}^{d}(GF^{-1})_{jk}\Phi_{\theta^\star,k}+c_j.\]
        Since $\E_{\theta^\star}[\Phi_{\theta^\star,j}]=0$ and $\E_{\theta^\star}[\partial_{\theta_j}\log q_{\theta^\star}]=0$, taking the expectation of this gives $c_j=0$,
        which proves $(1)\Rightarrow(3)$.

    $(3)\Rightarrow(1)$:
        Suppose there exists a matrix $\Lambda\in\R^{d\times d}$ such that
        \[\partial_{\theta_j}\log q_{\theta^\star}=\sum_{k=1}^d\Lambda_{jk}\Phi_{\theta^\star,k}.\]
        Applying $h\mapsto \A_{\theta^\star}(\nabla_x h)$ on both sides, we have
        \begin{align*}
            \A_{\theta^\star}(\nabla_x\partial_{\theta_j}\log q_{\theta^\star})
            &=\sum_{k=1}^d\Lambda_{jk}\A_{\theta^\star}(\nabla_x\Phi_{\theta^\star,k})\\
            &=-\sum_{k=1}^d\Lambda_{jk}\partial_{\theta_{j}}\log q_{\theta^\star},
        \end{align*}
        which implies that the estimating equation of the score matching estimator coincides with the MLE. Thus, 
        the score matching estimator is asymptotically efficient, implying $(3)\Rightarrow(1)$.

    $(3)\Rightarrow(2)$:
        Suppose there exists a matrix $\Lambda\in\R^{d\times d}$ such that
        \[\partial_{\theta_j}\log q_{\theta^\star}=\sum_{k=1}^d\Lambda_{jk}\Phi_{\theta^\star,k}.\]
        Fix $u^\star\colon\R^d\to\R$ that is $W$-orthogonal to $\nabla_x\partial_{\theta_j}\log q_{\theta^\star}$ for $j=1,\dots,d$ arbitrarily.
        Combining the assumption with Lemma~\ref{lem:connect_two_innerprod}, we have
        \begin{align*}
            \E_{\theta^\star}[(\A_{\theta^\star}u^\star)\A_{\theta^\star}(\nabla_x\partial_{\theta_j}\log q_{\theta^\star})]
            &=\sum_{k=1}^d\Lambda_{jk}\E_{\theta^\star}[(\A_{\theta^\star}u^\star)\A_{\theta^\star}(\nabla_x\Phi_{\theta^\star,j})]\\
            &=\sum_{k=1}^d\Lambda_{jk}\E_{\theta^\star}[\ang{u^\star,\nabla_x\partial_{\theta_k}\log q_{\theta^\star}}]\\
            &=0,
        \end{align*}
        which shows $u^\star$ is $\A_{\theta^\star}$-orthogonal to $\nabla_x\partial_{\theta_j}\log q_{\theta^\star}$ and completes $(3)\Rightarrow(2)$.
    
    $(2)\Rightarrow(1)$:
        Suppose for any element $u^\star$ that is $W$-orthogonal to $\nabla_x\partial_{\theta_j}\log q_{\theta^\star}$, it is also $\A_{\theta^\star}$-orthogonal to $\nabla_x\partial_{\theta_j}\log q_{\theta^\star}$ for $j = 1,\dots,d$.
        Let $u_j^\star$ be the $W$-orthogonal elements defined by \eqref{eq:wasserstein_score_orth}.
        By Lemma~\ref{lem:connect_two_innerprod}, we have
        \[\E_{\theta^\star}[(\A_{\theta^\star}u_j^\star)\A_{\theta^\star}(\nabla_x\Phi_{\theta^\star,k})]=0,\quad k=1,\dots,d.\]
        By Theorem~\ref{thm:smom_decomp}, we can decompose the asymptotic variance of the MLE:
        \[\AVar\left[\hat\theta_\mathrm{MLE}\right]=\AVar\left[\hat\theta_\mathrm{SM}\right]+G^{-1}(T+S+S^\top)G^{-1},\]
        where $S\in\R^{d\times d}$ and $T\in\R^{d\times d}$ are the inner produce matrices whose $(j,k)$-entries are defined by 
        \begin{align*}
            S_{jk}&\coloneqq\E_{\theta^\star}[\A_{\theta^\star}u_{j}^\star\A_{\theta^\star}(\nabla_x\partial_{\theta_k}\log q_{\theta^\star})],\\
            T_{jk}&\coloneqq\E_{\theta^\star}[(\A_{\theta^\star}u_j^\star)(\A_{\theta^\star}u_k^\star)],
        \end{align*}
        respectively.
        Since $S=O$, we have $\AVar[\hat\theta_\mathrm{MLE}]-\AVar[\hat\theta_\mathrm{SM}]=G^{-1}TG^{-1}$. 
        However, we also have $\AVar[\hat\theta_\mathrm{SM}]-\AVar[\hat\theta_\mathrm{MLE}]=G^{-1}TG^{-1}$ by Lemma~\ref{lem:gap_mle_sm}, so we finally obtain $\AVar[\hat\theta_\mathrm{SM}]=\AVar[\hat\theta_\mathrm{MLE}]$, which implies that the score matching estimator is asymptotically efficient and concludes $(2)\Rightarrow(1)$.
\end{proof}

\section{Main results on more general domains}
\label{sec:main_result_M}

In this appendix, we give an extension of our main results on Riemannian manifolds.

\subsection{Preliminaries}

We begin with preparing notations, review a generalization of score matching \citep{Williams2022}, and introducing a Stein operator we employ in the Appendix.

Let $\M$ be an oriented and connected Riemannian manifold with corners, and let $\mathrm{d}x$ be the volume form given by the Riemannian metric on $\M$.
Let $\partial\M$ be the boundary of $\M$, and let $N$ and $\mathrm{d}s$ denote the unit outward normal vector field and volume form on $\partial\M$, respectively. 
The Riemannian metric on $\M$ and the induced norm on the tangent space $T_x\M$ are denoted by $\ang{\,\cdot\,,\,\cdot\,}$ and $\norm{\,\cdot\,}$, respectively.

Let $C^\infty(\M)$ be the space of smooth functions on $\M$, and let $\X^\infty(\M)$ be the space of smooth vector fields on $\M$.
The gradient operator on $\M$ is denoted by $\nabla_\M$, and the divergence of a vector field $f\in\X^\infty(\M)$ is denoted by $\divm{f}$.

Let the parameter space $\Theta\subset\R^d$ be an open set, and let $q_{\theta}$ be a  probability density on $\M$ with respect to $\mathrm{d}x$ such that $q_\theta>0$ almost everywhere.
We assume that the model $\{q_{\theta}\mid\theta\in\Theta\}$ is identifiable; that is, $q_{\theta}=q_{\theta'}$ if and only if $\theta=\theta'$.
We also assume that $q_{\theta}$ and test functions $f_{\theta,j}\in\X^\infty(\M)$ for $j=1,\dots,d$ are smooth with respect to both $\theta$ and $x$.

Let $X_1,\dots,X_n$ be i.i.d.~random variables from $q_{\theta^\star}$.
Fix a weight function $w\in C^\infty(\M)$ that is strictly positive almost everywhere.
The case $w=1$ and $\M=\R^d$ corresponds to the main text.
The weighted score matching estimator $\hat\theta_\mathrm{wSM}$ is given as a solution to the estimating equations:
\[\frac{1}{n}\sum_{i=1}^n \partial_{\theta_j}H_w(X_i,\theta)=0,\quad j=1,\dots,d,\]
where the weighted Hyvärinen score $H_w$ is defined by $H_w(x,\theta)\coloneqq w(x)\norm{\scorem_\theta(x)}^2/2+\divm(w(x)\scorem_\theta(x))$.
Under the condition $\int_\M \divm(wq_{\theta^\star}\scorem_\theta)\d x=0$, we have the relation
\[\E_{\theta^\star}[H_w(X,\theta)]=\E_{\theta^\star}[w\norm{\scorem_{\theta^\star}-\scorem_{\theta}}^2]/2+c,\]
where $X\sim q_{\theta^\star}$ and $c$ is a constant independent of $\theta$.
Thus, minimizing $\E_{\theta^\star}[H_w(X,\theta)]$ over $\Theta$ is equivalent to minimizing $\E_{\theta^\star}[w\norm{\scorem_{\theta^\star}-\scorem_{\theta}}^2]$.

We employ the following weighted Stein operator:
\[\A_{\theta}^wf\coloneqq \frac{\divm(q_{\theta}wf)}{q_{\theta}},\quad f\in\X^\infty(\M).\]
For a test function $f$ satisfying the condition $\int_\M \divm(q_{\theta^\star}wf)\d x=0$, the identity $\E_{\theta^\star}[\A^w_{\theta^\star}f]=0$ holds.
The SMoM estimator based on the test functions $f_{\theta,1},\dots,f_{\theta,d}$ is given as a solution to the estimating equations $n^{-1}\sum_{i=1}^n\A^w_{\theta}f_{\theta,j}(X_i)=0$ for $j=1,\dots,d$.

The following lemma states that the weighted score matching estimator is an SMoM estimator. 

\begin{lem}[Lemma 3.2 of \cite{Kume2026}]
    \label{lem:smom_sm_M}
    The SMoM estimator based on the test functions $f_{\theta,j}\coloneqq\nabla_{\M}\partial_{\theta_j}\log q_{\theta}, \enspace j=1,\dots,d$ is the weighted score matching estimator.
\end{lem}
\begin{proof}   
    Since $f_{\theta,j}=\nabla_{\M}\partial_{\theta_j}\log q_\theta$ is a solution of the equation $\A^w_{\theta}f_{\theta,j}(x)=\partial_{\theta_j}H_w(x,\theta)$,
    the estimating equations of weighted score matching coincide with that of SMoM based on the test functions.
\end{proof}

\begin{remark}[The weight function $w$]
    The weight function $w$ assists in satisfying the boundary condition.
    In $\M=\R^d$, the condition $\int_\M \divm(wq_{\theta^\star}f)\d x=0$ is satisfied for $f\in\X^\infty(\M)$ if we have $\lim_{\|x\|\to\infty}w(x)q_{\theta^\star}(x)f(x)=0$.
    If $\M$ is compact and $w|_{\partial\M}=0$, by Stokes' theorem, we have
    \[\int_\M\divm(wq_{\theta^\star}f)\d x=\int_{\partial\M}wq_{\theta^\star}\ang{f,N}\d s,\]
    which automatically vanishes.
\end{remark}

\subsection{Regularity conditions for the generalization}
\label{sec:regularity_M}
We assume the following conditions for the model $\{q_\theta\mid\theta\in\Theta\}$ and test functions $f_\theta$:

\begin{enumerate}
    \item[C1.] $\E_{\theta}[(\A^w_{\theta}f_\theta)^2]<\infty$ and $\E_{\theta}[\A^w_{\theta}f_{\theta}]=0$.
    \item[C2.] $\E_{\theta}[\A^w_{\theta}(\partial_{\theta_j}f_{\theta})]=0,\quad j=1,\dots,d$.
    \item[C3.] $\E_{\theta}[|w\ang{f_\theta,\nabla_\M\partial_{\theta_j}\log q_{\theta}}|]<\infty,\quad j=1,\dots,d$.
\end{enumerate}
For the SMoM estimator based on the test functions $f_{\theta,1},\dots,f_{\theta,d}$, we assume the following conditions:
\begin{enumerate}
    \item[C4.] $\A^w_{\theta^\star}f_{\theta^\star,1},\dots,\A^w_{\theta^\star}f_{\theta^\star,d}$ are linearly independent.
    \item[C5.] The matrix $G_\mathrm{SMoM}\in\R^{d\times d}$ whose $(j,k)$-th entry is defined by \[(G_\mathrm{SMoM})_{jk}\coloneqq\E_{\theta^\star}\Big[\partial_{\theta_k}\A^w_{\theta}f_{\theta,j}\Big|_{\theta=\theta^\star}\Big]\] is invertible.
    \item[C6.] $\hat\theta_\mathrm{SMoM}$ is consistent and asymptotic linear; that is, $\hat\theta_\mathrm{SMoM}-\theta^\star$ has the following representation: \begin{align}\hat\theta_\mathrm{SMoM}-\theta^\star=-G_\mathrm{SMoM}^{-1}\frac{1}{n}\sum_{i=1}^{n}\begin{pmatrix}\A^w_{\theta^\star}f_{\theta^\star,1}(X_i)\\\vdots\\\A^w_{\theta^\star}f_{\theta^\star,d}(X_i)\end{pmatrix} + \smallop.\label{eq:SMoM linear M}\end{align}
\end{enumerate}

\subsection{Generalized canonical decomposition of SMoM estimators}

We then present extensions of our main result.
Under the regularity conditions and Lemma~\ref{lem:smom_sm_M}, the weighted score matching estimator $\hat\theta_\mathrm{wSM}$ has the following asymptotic linear representation:
\begin{align}
\hat\theta_\mathrm{wSM}-\theta^\star
    &=-G^{-1}\frac{1}{n}\sum_{i=1}^{n}\begin{pmatrix}\A^w_{\theta^\star}(\nabla_{\M}\partial_{\theta_1}\log q_{\theta^\star})(X_i)\\\vdots\\\A^w_{\theta^\star}(\nabla_{\M}\partial_{\theta_d}\log q_{\theta^\star})(X_i)\end{pmatrix}+ \smallop\label{eq:SM linear M}
\end{align}
where $G\in\R^{d\times d}$ are the matrices whose $(j,k)$-th entries are defined by
\[G_{jk}\coloneqq\E_{\theta^\star}\Big[\partial_{\theta_k}\A^w_{\theta}(\nabla_{\M}\partial_{\theta_j}\log q_\theta)\Big|_{\theta=\theta^\star}\Big].\]

We first introduce weighted $W$-orthogonality and $\A^w_{\theta^\star}$-orthogonality.
\begin{dfn}
    A test function $f$ is called weighted $W$-orthogonal to a test function $g$ if we have $\E_{\theta^\star}[w\ang{g,f}]=0$.
    A test function $f$ is called $\A^w_{\theta^\star}$-orthogonal to a test function $g$ if we have $\E_{\theta^\star}[(\A^w_{\theta^\star}g)(\A^w_{\theta^\star}f)]=0$.
\end{dfn}

The following theorem, which is a generalized result of Theorem~\ref{thm:smom_decomp}, states that the asymptotic linear representation of the SMoM estimator after centering the weighted score matching estimator.
It is proven by combining the asymptotic linear representation \eqref{eq:SMoM linear M}-\eqref{eq:SM linear M} and Lemma~\ref{lem:G_mat_M}, and the proof follows the same steps as that of Theorem~\ref{thm:smom_decomp}, so it is omitted.

\begin{thm}[Generalized canonical decomposition of SMoM estimator]
    \label{thm:smom_decomp_M}
    Under the regularity conditions, the asymptotic linear representation of $\hat\theta_\mathrm{SMoM}$ is decomposed as follows:
    \[\hat\theta_\mathrm{SMoM}-\theta^\star=\Big(\hat\theta_\mathrm{wSM}-\theta^\star\Big)-G^{-1}\frac{1}{n}\sum_{i=1}^{n}\begin{pmatrix}\A^w_{\theta^\star}u_1^\star(X_i)\\\vdots\\\A^w_{\theta^\star}u_d^\star(X_i)\end{pmatrix} + \smallop , \]
    where each $u_{j}^\star\in\X^\infty(\M)$ is weighted $W$-orthogonal to the test function corresponding to weighted score matching:
    \[
    \E_{\theta^\star}[w\ang{u_j^\star,\nabla_{\M}\partial_{\theta_k}\log q_{\theta^\star}}]=0, \quad k=1,\dots,d.
    \]
\end{thm}

The following lemma tells us that $G_\mathrm{SMoM}$ is an inner product matrix, and in particular, $G$ is a Gram matrix.
\begin{lem}
    \label{lem:G_mat_M}
    Under the regularity conditions, for the test functions $f_{\theta,1},\dots,f_{\theta,d}$, we have \[\E_{\theta^\star}\Big[\partial_{\theta_k}\A^w_{\theta}f_{\theta,j}\Big|_{\theta=\theta^\star}\Big]=\E_{\theta^\star}[w\ang{f_{\theta^\star,j},\nabla_{\M}\partial_{\theta_k}\log q_{\theta^\star}}],\quad j,k=1,\dots,d.\]
    In particular, we have $G_{jk}=\E_{\theta^\star}\left[w\ang{\nabla_{\M}\partial_{\theta_j}\log q_{\theta^\star},\nabla_{\M}\partial_{\theta_k}\log q_{\theta^\star}}\right]$.
\end{lem}
\begin{proof}
    Observe the relation
    \[\partial_{\theta_k}\A^w_\theta f_{\theta,j}=\A^w_{\theta}(\partial_{\theta_k}f_{\theta,j})+w\ang{f_{\theta,j},\nabla_{\M}\partial_{\theta_k}\log q_\theta}.\]
    Together with $\E_{\theta^\star}[\A^w_{\theta^\star}(\partial_{\theta_k}f_{\theta^\star,j})]=0$ under the regularity conditions, taking the expectation of this yields $\E_{\theta^\star}[\partial_{\theta_k}\A^w_{\theta}f_{\theta,j}|_{\theta=\theta^\star}]=\E_{\theta^\star}[w\ang{f_{\theta^\star,j},\nabla_{\M}\partial_{\theta_k}\log q_{\theta^\star}}]$, which completes the proof.
\end{proof}

\begin{exm}[Exponential families on Riemannian manifolds]
Consider an exponential family whose density is defined by
\begin{equation}
    \label{eq:exp_fam_M}
    q_{\theta}(x)=\frac{1}{Z(\theta)}\exp\Big(t(x)^{\top}\theta+b(x)\Big),\enspace Z(\theta)=\int_\M\exp\Big(t(x)^{\top}\theta+b(x)\Big)\d{x},
\end{equation}
where $t_j\in C^\infty(\M),\enspace j=1,\dots,d$ is a sufficient statistics and $b\in C^\infty(\M)$ is a base measure.
The SMoM estimator $\hat\theta_\mathrm{SMoM}$ based on (parameter-independent) test functions $f_1,\dots,f_d$ has the closed-form solution:
\[\hat\theta_\mathrm{SMoM}=-G_{\mathrm{SMoM},n}^{-1}\frac{1}{n}\sum_{i=1}^{n}\begin{pmatrix}\divm(wf_1)(X_i) + w(X_i)\ang{f_1(X_i),\nabla_\M b(X_i)}\\\vdots\\\divm(wf_d)(X_i)+w(X_i)\ang{f_d(X_i),\nabla_\M b(X_i)}\end{pmatrix},\]
where $G_{\mathrm{SMoM},n}\in\R^{d\times d}$ is the empirical inner product matrix whose $(j,k)$-th entry is defined by $(G_{\mathrm{SMoM},n})_{jk} \coloneqq n^{-1}\sum_{i=1}^{n}w(X_i)\ang{f_j(X_i),\nabla_\M t_k(X_i)}$.
The score matching estimator $\hat\theta_\mathrm{wSM}$ is the SMoM estimator based on the test functions $f_{j}=\nabla_\M t_j$, which also has the closed-form solution \citep{Mardia2016,Scealy2023}:
\[\hat\theta_\mathrm{wSM}=-G_{n}^{-1}\frac{1}{n}\sum_{i=1}^{n}\begin{pmatrix}\divm(w\nabla_\M t_1)(X_i) + w(X_i)\ang{\nabla_\M t_1(X_i),\nabla_\M b(X_i)}\\\vdots\\\divm(w\nabla_\M t_d)(X_i)+w(X_i)\ang{\nabla_\M t_d(X_i),\nabla_\M b(X_i)}\end{pmatrix},\]
where $G_{n}\in\R^{d\times d}$ is the empirical inner product matrix whose $(j,k)$-th entry is defined by $(G_{n})_{jk} \coloneqq n^{-1}\sum_{i=1}^{n}w(X_i)\ang{\nabla_\M t_j(X_i),\nabla_\M t_k(X_i)}$.
Consider the test functions $f_j=\nabla_\M t_j + u_j^\star$ for $j=1,\dots,d$, where $u_j^\star$ is weighted $W$-orthogonal to $\nabla_\M t_k$ for $k=1,\dots,d$.
The SMoM estimator based on this test functions can be decomposed as follows:
\begin{align*}
    \hat\theta_\mathrm{SMoM}-\theta^\star
    &=-\big(G_n + (G_{\mathrm{SMoM},n}-G_n)\big)^{-1}\frac{1}{n}\sum_{i=1}^{n}\begin{pmatrix}\A^w_{\theta^\star}(\nabla_\M t_1)(X_i)+\A^w_{\theta^\star}u_1^\star(X_i)\\\vdots\\\A^w_{\theta^\star}(\nabla_\M t_d)(X_i)+\A^w_{\theta^\star}u_d^\star(X_i)\end{pmatrix}\\
    &=\Big(\hat\theta_\mathrm{wSM}-\theta^\star\Big)-G^{-1}\frac{1}{n}\sum_{i=1}^{n}\begin{pmatrix}\A^w_{\theta^\star}u_1^\star(X_i)\\\vdots\\\A^w_{\theta^\star}u_d^\star(X_i)\end{pmatrix}+\smallop.
\end{align*}
\end{exm}

\subsection{Improving asymptotic variance of weighted score matching estimator}
We next construct an SMoM estimator improving upon the weighted score matching estimator in the asymptotic variance.
Fix $\theta_0\in\Theta$ and $\tilde{v}_\alpha\in\X^\infty(\M),\enspace\alpha=1,\dots,K$ arbitrarily. 
Using the test function $\tilde{v}_\alpha$, we define $v_{\theta_0,\alpha}$ by the following orthogonalization procedure:
\begin{equation}
    \label{eq:orthogonalize_M}
    v_{\theta_0,\alpha}\coloneqq \tilde{v}_\alpha - \sum_{j=1}^{d}\left(F_{\theta_0}G_{\theta_0}^{-1}\right)_{\alpha j}\nabla_{\M}\partial_{\theta_j}\log q_{\theta_0},
\end{equation}
where $F_{\theta_0}\in\R^{K\times d}$ and $G_{\theta_0}\in\R^{d\times d}$ are inner product matrices whose $(\alpha,j)$-th entry and $(j,k)$-th entry are defined by 
\begin{align*}
    (F_{\theta_0})_{\alpha j}&\coloneqq \E_{\theta_0}\left[w\ang{\tilde{v}_\alpha,\nabla_{\M}\partial_{\theta_j}\log q_{\theta_0}}\right],\\
    (G_{\theta_0})_{jk}&\coloneqq \E_{\theta_0}\left[w\ang{\nabla_{\M}\partial_{\theta_j}\log q_{\theta_0},\nabla_{\M}\partial_{\theta_k}\log q_{\theta_0}}\right],
\end{align*}
respectively.
Observe the relation
\[\E_{\theta_0}[w\ang{\nabla_{\M}\partial_{\theta_j}\log q_{\theta_0},v_{\theta_0,\alpha}}]=0,\quad j=1,\dots,d.\]
Denote by $\hat\theta[\theta_0]$ the SMoM estimator based on the following test functions:
\begin{equation}
    \label{eq:testfunction_M} 
    f_{\theta,j}\coloneqq\nabla_{\M}\partial_{\theta_j}\log q_{\theta} - \sum_{\alpha=1}^{K}\left(S_{\theta_0}T_{\theta_0}^{-1}\right)_{j\alpha}v_{\theta_0,\alpha},\quad j=1,\dots,d,
\end{equation}
where $S_{\theta_0}\in\R^{d\times K}$ and $T_{\theta_0}\in\R^{K\times K}$ are inner product matrices whose $(j,\alpha)$-th entry and $(\alpha,\beta)$-th entry are defined by
\begin{align*}
    (S_{\theta_0})_{j\alpha}&\coloneqq\E_{\theta_0}\left[\A^w_{\theta_0}(\nabla_{\M}\partial_{\theta_j}\log q_{\theta_0})\A^w_{\theta_0}v_{\theta_0,\alpha}\right],\\
    (T_{\theta_0})_{\alpha\beta}&\coloneqq\E_{\theta_0}\left[\left(\A^w_{\theta_0}v_{\theta_0,\alpha}\right)\left(\A^w_{\theta_0}v_{\theta_0,\beta}\right)\right].
\end{align*}
For the matrices $F_{\theta_0},G_{\theta_0},S_{\theta_0},T_{\theta_0}$, we make the following assumption.
\begin{asm}
    \label{asm:improve_M}
    $F_{\theta_0},G_{\theta_0},S_{\theta_0},T_{\theta_0}$ are continuous at $\theta_0=\theta^\star$.
\end{asm}

The following theorem, which is generalized result of Theorem~\ref{thm:improve}, tells us that $\hat\theta[\theta^\star]$ improves the asymptotic variance of $\hat\theta_\mathrm{wSM}$.
Furthermore, such improvement remains even when  $\hat\theta_0$ is plugged in to $\theta^\star$.
\begin{thm}
    \label{thm:improve_M}
    Under the regularity conditions, we have
    \[\AVar\left[\hat\theta_{\mathrm{wSM}}\right]-\AVar\left[\hat\theta[\theta^\star]\right]
    =G^{-1}S_{\theta^\star}T_{\theta^\star}^{-1}S_{\theta^\star}^{\top}G^{-1}\succeq0.\]
    The equality $\AVar[\hat\theta_{\mathrm{wSM}}]=\AVar[\hat\theta[\theta^\star]]$ holds if and only if $S_{\theta^\star}=O$.
    Furthermore, under Assumption~\ref{asm:improve},
    for an estimator $\hat\theta_0$ satisfying $\hat\theta_0-\theta^\star=\largeop$, we have \[\hat\theta[\hat\theta_0]-\theta^\star=\hat\theta[\theta^\star]-\theta^\star+\smallop.\]
\end{thm}
\begin{proof}
    We first show that $\AVar[\hat\theta_{\mathrm{wSM}}]-\AVar[\hat\theta[\theta^\star]]=G^{-1}S_{\theta^\star}T_{\theta^\star}^{-1}S_{\theta^\star}^{\top}G^{-1}$.
    Combining with Theorem~\ref{thm:smom_decomp_M} and \eqref{eq:testfunction_M}, we can decompose the asymptotic linear representation of $\hat\theta[\theta^\star]$:
    \[\hat\theta[\theta^\star]-\theta^\star=\Big(\hat\theta_\mathrm{wSM}-\theta^\star\Big)+G^{-1}S_{\theta^\star}T_{\theta^\star}^{-1}\frac{1}{n}\sum_{i=1}^{n}\Big(\A^w_{\theta^\star}v_{\theta^\star,1},\dots,\A^w_{\theta^\star}v_{\theta^\star,K}\Big)^\top + \smallop.\]
    Thus, we can also decompose the asymptotic variance of $\hat\theta[\theta^\star]$ as follows:
    \[\AVar\left[\hat\theta[\theta^\star]\right]=\AVar\left[\hat\theta_\mathrm{wSM}\right]+G^{-1}(-S_{\theta^\star}T_{\theta^\star}^{-1}S_{\theta^\star}^\top)G^{-1}.
    \]
    Since $T_{\theta^\star}$ and $G$ is positive definite, we have $\AVar[\hat\theta[\theta^\star]]=\AVar[\hat\theta_\mathrm{wSM}]$ if and only if $S_{\theta^\star}=0$ holds. 

    Next, we show $\hat\theta[\hat\theta_0]-\theta^\star=\hat\theta[\theta^\star]-\theta^\star+\smallop$.
    Using Theorem~\ref{thm:smom_decomp_M}, \eqref{eq:orthogonalize_M}, and \eqref{eq:testfunction_M}, we can decompose the asymptotically linear representation of $\hat\theta[\hat\theta_0]$:
    \begin{align}
        \hat\theta[\hat\theta_0]-\theta^\star
        &=\Big(\hat\theta_\mathrm{wSM}-\theta^\star\Big)+G^{-1}S_{\hat\theta_0}T_{\hat\theta_0}^{-1}\frac{1}{n}\sum_{i=1}^{n}\begin{pmatrix}\A^w_{\theta^\star}v_{\hat\theta_0,1}\\\vdots\\\A^w_{\theta^\star}v_{\hat\theta_0,K}\end{pmatrix} + \smallop\nonumber\\
        &=\Big(\hat\theta_\mathrm{wSM}-\theta^\star\Big)+G^{-1}S_{\hat\theta_0}T_{\hat\theta_0}^{-1}\frac{1}{n}\sum_{i=1}^{n}\begin{pmatrix}\A^w_{\theta^\star}\tilde{v}_1\\\vdots\\\A^w_{\theta^\star}\tilde{v}_K\end{pmatrix}\nonumber\\
        &\hphantom{=}-G^{-1}S_{\hat\theta_0}T_{\hat\theta_0}^{-1}F_{\hat\theta_0}G_{\hat\theta_0}^{-1}\frac{1}{n}\sum_{i=1}^{n}\begin{pmatrix}\A^w_{\theta^\star}(\nabla_{\M}\partial_{\theta_1}\log q_{\hat\theta_0})\\\vdots\\\A^w_{\theta^\star}(\nabla_{\M}\partial_{\theta_d}\log q_{\hat\theta_0})\end{pmatrix} + \smallop.\label{eq:decompose_smom_plg}
    \end{align}
    Consider Taylor expansion of $\partial_{\theta_j}\log q_{\hat\theta_0}$.
    Since the identity $\E_{\theta^\star}[\A^w_{\theta^\star}(\nabla_\M\partial_{\theta_j}\partial_{\theta_k}\log q_{\theta^\star})]=0$ holds under the regularity conditions, we have
    \begin{align}        
        &\frac{1}{n}\sum_{i=1}^{n}\A^w_{\theta^\star}(\nabla_{\M}\partial_{\theta_j}\log q_{\hat\theta_0})\nonumber\\
        &=\frac{1}{n}\sum_{i=1}^{n}\left(\A^w_{\theta^\star}(\nabla_{\M}\partial_{\theta_j}\log q_{\theta^\star})(X_i)+\sum_{k=1}^d\A^w_{\theta^\star}\left(\nabla_{\M}\partial_{\theta_j}\partial_{\theta_k}\log q_{\theta^\star}\right)(X_i)(\hat\theta_0-\theta^\star)_k\right) + \smallop[\norm{\hat\theta_0-\theta^\star}]\nonumber\\
        &=\frac{1}{n}\sum_{i=1}^{n}\A^w_{\theta^\star}(\nabla_{\M}\partial_{\theta_j}\log q_{\theta^\star})(X_i)+\smallop[1]\largeop[n^{-1/2}]+\smallop.\label{eq:fisherscore_expansion}
    \end{align}
    Substituting \eqref{eq:fisherscore_expansion} to \eqref{eq:decompose_smom_plg} and using continuous mapping theorem, we finally obtain
    \begin{align*}
        \hat\theta[\hat\theta_0]-\theta^\star
        &=\Big(\hat\theta_\mathrm{wSM}-\theta^\star\Big)+G^{-1}S_{\theta^\star}T_{\theta^\star}^{-1}\frac{1}{n}\sum_{i=1}^{n}\begin{pmatrix}\A^w_{\theta^\star}v_{\theta^\star,1}\\\vdots\\\A^w_{\theta^\star}v_{\theta^\star,K}\end{pmatrix} + \smallop\\
        &=\hat\theta[\theta^\star]-\theta^\star+\smallop,
    \end{align*}
    which completes the proof.
\end{proof}

Let $U_{\theta_0}\in\R^{d\times d}$ be the inner matrix whose $(j,k)$-th entry defined by
\[(U_{\theta_0})_{jk}\coloneqq\E_{\theta_0}\left[\A^w_{\theta_0}(\nabla_{\M}\partial_{\theta_j}\log q_{\theta_0})\A^w_{\theta_0}(\nabla_{\M}\partial_{\theta_k}\log q_{\theta_0})\right].\]
The asymptotic variance of the weighted score matching estimator is given by $\AVar[\hat\theta_\mathrm{wSM}]=G^{-1}U_{\theta^\star}G^{-1}$.
As the estimate of asymptotic relative efficiency $\AVar[\hat\theta[\hat\theta_0]]_{jj}/\AVar[\hat\theta_\mathrm{wSM}]_{jj}$, we employ
\begin{equation}
    \label{eq:estimate_improvement_M}
    1-\frac{\left(G_{\hat\theta_0}^{-1}S_{\hat\theta_0}T_{\hat\theta_0}^{-1}S_{\hat\theta_0}^{\top}G_{\hat\theta_0}^{-1}\right)_{jj}}{\left(G_{\hat\theta_\mathrm{wSM}}^{-1}U_{\hat\theta_\mathrm{wSM}}G_{\hat\theta_\mathrm{wSM}}^{-1}\right)_{jj}},\quad j=1,\dots,d.
\end{equation}

\section{Additional numerical experiments}
\label{sec:simu_appl_M}
In this appendix, we provide additional numerical experiments illustrating the SMoM estimator constructed in Theorem~\ref{thm:improve_M}.
We focus on estimating the parameters $\theta$.
The estimates $\hat\theta_\mathrm{wSM},\hat\theta[\theta^\star]$, and $\hat\theta[\hat\theta_\mathrm{wSM}]$ are calculated using an i.i.d.~sample of size $n$ drawn from the distribution with $\theta^\star$.
This procedure is iterated in 1000 times, and the MSE of each estimates are calculated.
The estimates of asymptotic relative efficiency given by \eqref{eq:estimate_improvement_M} are also calculated.
$\tilde{v}_1,\dots,\tilde{v}_K$ is constructed using neural networks.
Specifically, $\tilde{v}_\alpha$ is implemented as a neural network with five hidden layers, each consisting of three nodes, and $p$-dimensional input and output layers.
We employ $\tanh$ as the activation function, and all parameters are randomly initialized with values drawn from $N(0,1)$.
To ensure the output is a valid vector field on $\M$, the output vector is projected onto the tangent space $T_x\M$ via the orthogonal projection $\mathrm{P}_x$.
Since the estimates $\hat\theta[\theta^\star],\hat\theta[\hat\theta_\mathrm{wSM}]$ depend on the choice of $\tilde{v}_1,\dots,\tilde{v}_K$, we evaluate the performance over 10 different pairs of initializations.
The sample size $n$ is set to 100.
The expectations are approximated via Monte Carlo integration with 1000 samples.

\subsection{Matrix Bingham distribution}
Consider the Stiefel manifold $\mathbb{V}_{k,p}\coloneqq\{X\in\mathbb{R}^{p\times k}\mid X^\top X=I_k\}$.
The orthogonal projection is given by $\mathrm{P}_X(Z)=Z-X(X^\top Z+Z^\top X)/2$.
The matrix Bingham distribution \citep{Chikuse2003} has the following density:
\[q_{\theta}(X)=\frac{1}{Z(A)}\exp\left(\operatorname{tr}(X^\top AX)\right),\]
where $\theta=A$, $A\in\R^{p\times p}$ is a symmetric matrix.
For identifiability, we fix $A_{pp}=0$.

We set to $(p,k)=(3,2)$, $A^\star=\operatorname{diag}(1,1,0)$.
Since $\mathbb{V}_{3,2}$ is compact manifold without boundary, we use the weight function $w(X)=1$.
The number $K$ of orthogonal elements varies in $\{3,6,12,24\}$.
Table~\ref{tab:mbingham_improvement} shows that both of $\hat\theta[\theta^\star],\hat\theta[\hat\theta_\mathrm{SM}]$ does not improve the variance of the score matching estimator for all $K$.
Figure~\ref{fig:mbingham_improvement_estimate} shows that the estimate of asymptotic relative efficiency concentrate near 1; that is, it implies that the score matching estimator has little room for improvement.

\begin{table}[H]
    \centering
    \caption{MSE ratio of matrix Bingham distribution for $\hat\theta[\hat\theta_\mathrm{SM}]$, $\hat\theta[\theta^\star]$ relative to $\hat\theta_\mathrm{SM}$. The results are represented in the format: median (min, max) across different pairs of $\tilde{v}_\alpha$. A value smaller than 1 indicates better performance.}
    \begin{tabular}{lcll}
  \hline
  && $\hat\theta[\hat\theta_\mathrm{SM}]$ & $\hat\theta[\theta^\star]$ \\
  \hline
   $K=3$ & $A_{11}$ & $1.006 \,(0.990,1.016)$ & $1.006 \,(1.000,1.016)$ \\
   & $A_{22}$ & $1.007 \,(0.997,1.018)$ & $1.007 \,(1.005,1.022)$ \\
   & $A_{12}$ & $1.008 \,(1.000,1.016)$ & $1.008 \,(0.995,1.015)$ \\
   & $A_{13}$ & $1.006 \,(1.000,1.011)$ & $1.004 \,(0.993,1.015)$ \\
   & $A_{23}$ & $1.007 \,(1.000,1.014)$ & $1.009 \,(0.994,1.017)$ \\\hline
   $K=6$ & $A_{11}$ & $1.010 \,(1.001,1.015)$ & $1.012 \,(0.997,1.017)$ \\
   & $A_{22}$ & $1.014 \,(0.995,1.023)$ & $1.011 \,(1.005,1.018)$ \\
   & $A_{12}$ & $1.011 \,(0.996,1.016)$ & $1.011 \,(1.006,1.023)$ \\
   & $A_{13}$ & $1.011 \,(0.995,1.027)$ & $1.011 \,(0.993,1.018)$ \\
   & $A_{23}$ & $1.014 \,(0.999,1.032)$ & $1.017 \,(1.005,1.029)$ \\\hline
   $K=12$ & $A_{11}$ & $1.016 \,(1.003,1.025)$ & $1.023 \,(1.006,1.040)$ \\
   & $A_{22}$ & $1.020 \,(1.005,1.031)$ & $1.020 \,(0.998,1.026)$ \\
   & $A_{12}$ & $1.021 \,(1.005,1.047)$ & $1.026 \,(1.010,1.041)$ \\
   & $A_{13}$ & $1.021 \,(1.002,1.032)$ & $1.012 \,(1.004,1.038)$ \\
   & $A_{23}$ & $1.025 \,(1.014,1.032)$ & $1.033 \,(1.001,1.045)$ \\\hline
   $K=24$ & $A_{11}$ & $1.037 \,(1.011,1.051)$ & $1.046 \,(1.018,1.065)$ \\
   & $A_{22}$ & $1.039 \,(1.024,1.056)$ & $1.035 \,(1.016,1.055)$ \\
   & $A_{12}$ & $1.045 \,(1.023,1.065)$ & $1.038 \,(1.023,1.050)$ \\
   & $A_{13}$ & $1.041 \,(1.032,1.052)$ & $1.048 \,(1.027,1.095)$ \\
   & $A_{23}$ & $1.039 \,(1.026,1.072)$ & $1.051 \,(1.022,1.060)$ \\\hline
\end{tabular}

    \label{tab:mbingham_improvement}
\end{table}

\begin{figure}[H]
    \centering
    \includegraphics[width=\linewidth]{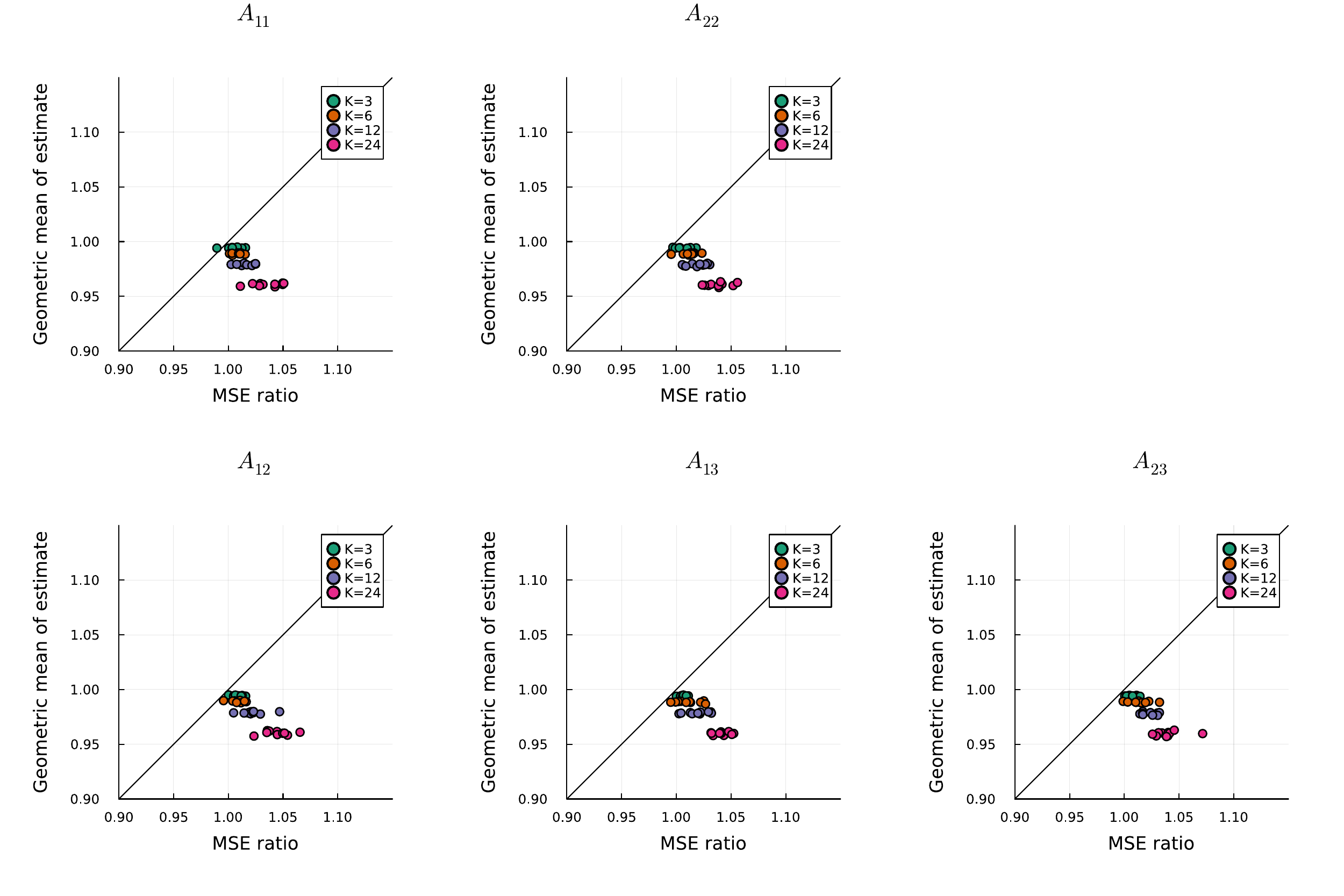}
    \caption{MSE ratio of matrix Bingham distribution for $\hat\theta[\hat\theta_\mathrm{SM}]$ relative to $\hat\theta_\mathrm{SM}$ versus the geometric mean of estimates given by \eqref{eq:estimate_improvement_M}. Values less than 1 on the horizontal axis indicate that corresponding SMoM estimator improves the variance of the score matching estimator. Points near the diagonal indicate that the estimate of the asymptotic relative efficiency is reliable.}
    \label{fig:mbingham_improvement_estimate}
\end{figure}

\end{document}